\documentclass{article}

\usepackage[english]{babel}
\usepackage[letterpaper,margin=3cm,marginparwidth=1.75cm]{geometry}

\usepackage{enumitem}

\usepackage{amsmath,amsfonts,amssymb,amsthm,graphics}
\usepackage{thmtools}
\usepackage{appendix}

\usepackage{hyperref}
\hypersetup{
    colorlinks=true,
    linkcolor=blue,
    citecolor=blue,
    urlcolor=blue,
    pdfborder={0 0 0}
}

\usepackage[utf8]{inputenc}
\usepackage{yfonts}
\usepackage{graphicx}
\usepackage{xcolor}
\usepackage{bbm}
\usepackage{wrapfig} 
\usepackage[font=sf, labelfont={sf,bf}, margin=1cm]{caption}

\usepackage[nameinlink,capitalise]{cleveref}
  \crefname{theorem}{Theorem}{Theorems}
  \crefname{conjecture}{Conjecture}{Conjectures}
  \crefname{thm}{Theorem}{Theorems}
  \crefname{thm*}{Theorem*}{Theorems}
  \crefname{lemma}{Lemma}{Lemmas}
  \crefname{lem}{Lemma}{Lemmas}
  \crefname{remark}{Remark}{Remarks}
  \crefname{prop}{Proposition}{Propositions}
  \crefname{fact}{Fact}{Facts}
  \crefname{notation}{Notation}{Notations}
  \crefname{claim}{Claim}{Claims}
  \crefname{defn}{Definition}{Definitions}
  \crefname{corollary}{Corollary}{Corollaries}
  \crefname{section}{Section}{Sections}
  \crefname{figure}{Figure}{Figures}
  \crefname{assumption}{Assumption}{Assumptions}

\newtheorem{thm}{Theorem}[section]
\newtheorem{thm*}{Theorem*}[section]

\newtheorem{claim}[thm]{Claim}
\newtheorem{lemma}[thm]{Lemma}
\newtheorem{corollary}[thm]{Corollary}
\newtheorem{prop}[thm]{Proposition}

\newtheorem{problem}[thm]{Problem}
\newtheorem{fact}[thm]{Fact}

\numberwithin{equation}{section}
\theoremstyle{definition}
\newtheorem{remark}[thm]{Remark}

\def\cS{\mathcal S}
\def\cT{\mathcal{T}}
\def\T{\mathbb{T}}
\def\Z{\mathbb{Z}}
\def\E{\mathbb{E}}
\def\1{\mathbf{1}}
\renewcommand{\P}{\mathbb{P}}
\def\Pr{\P}

\def\R{\mathbb R}
\def\eps{\varepsilon}
\def \cO {\mathcal O}
\def \cH {\mathcal H}
\def\cK {\mathcal{K}}

\title{A noisy min-max game on trees}

\author{Omer Angel\thanks{University of British Columbia. Email: angel@math.ubc.ca}
  \and Gourab Ray \thanks{University of Victoria. Email: gourabray@uvic.ca}
  \and Yinon Spinka\thanks{Tel Aviv University. Email: yinonspi@tauex.tau.ac.il}}

\begin{document}

\maketitle

\begin{abstract}
  We study a noisy version of a min-max type zero-sum game on the $d$-ary tree.
  Each edge of the tree is assigned an i.i.d.\ cookie, distributed uniformly on $\{+1,-1\}$.
  The game is played as follows: starting at the root, two players alternate turns in choosing a child to move to, with the game ending after each player took $n$ turns.
  Both players have full knowledge of the cookies on the whole tree.
  The cookies along the traversed edges are picked up and placed in a shared cookie jar.
  The first player's payoff is the sum of the cookies in the cookie jar, while the second player pays that sum.
  The value $V_n$ of the $n$-round game is the largest signed sum which can be guaranteed by the first player.
  We analyze the value $V_n$ and show that as $n \to \infty$, the value is tight for $d=2$, converges in distribution for $d \ge 3$ and converges almost surely for $d \ge 15$.
  Along the way, we prove various tightness and double exponential tail decay results.

  The analysis is a mix of percolation type arguments for large $d$, and iterations on distributions combined with interval arithmetic for small $d$.
  For $d=2$ we prove the existence of a continuum of fixed points for this iteration, highlighting surprising qualitative differences with the case of $d\ge3$.
  The question of convergence for $d=2$ remains open.
\end{abstract}

\section{Introduction}

Let $\T_d$ be the infinite $d$-ary tree.
Place an independent uniform \textbf{cookie} $X_e \in \{-1,1\}$ on each edge $e$ of $\T_d$.
The \textbf{$n$-round zero-sum game} is played as follows:
Starting at the root, two players alternate turns choosing a child to move to, until level $2n$ of the tree is reached (after $n$ rounds).
The \textbf{payoff} to player 1 from player 2 is the sum of the $2n$ cookies along the chosen path.
The objective of the first player is to maximize the payoff, while the objective of the second player is to minimize it.
The \textbf{value} of the $n$-round game, denoted $V_n$, is the largest payoff that the first player can guarantee (i.e., the payoff obtained when both players follow an optimal strategy).
Both players have full information of the cookies on the entire tree.
Note that $V_n$ is always an even integer.
See \cref{fig:tree_eg} for an example of a 2-round game with the associated values.

\begin{figure}[h]
  \centering
  \includegraphics[width=.95\textwidth]{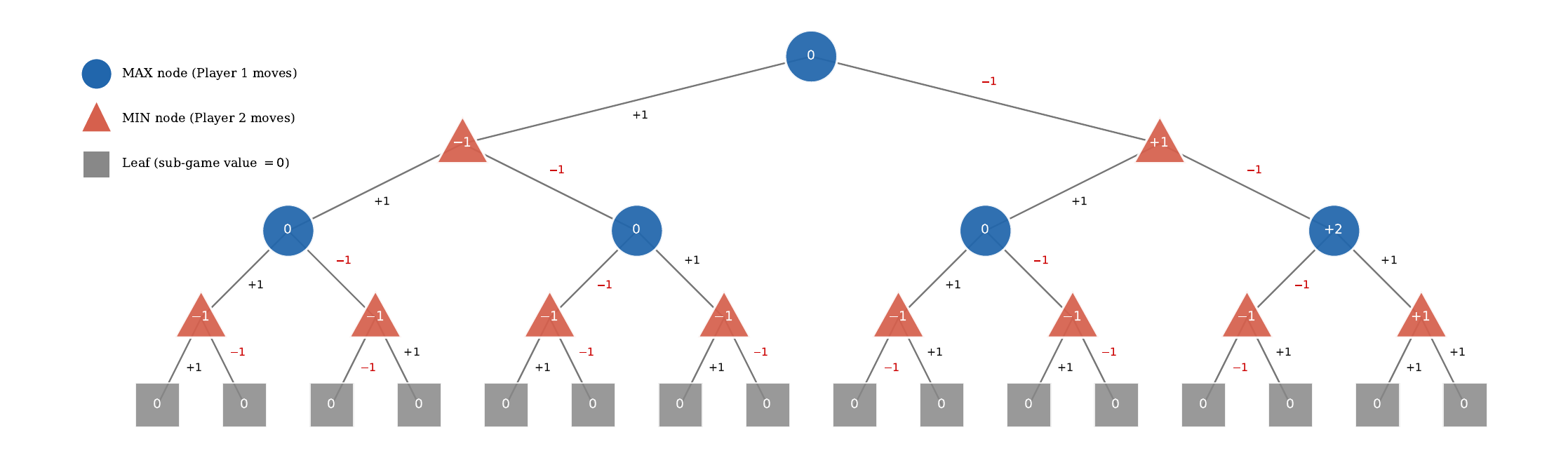}
  \caption{A 2 round game on the binary tree. Each edge is labeled with its cookie;
    each vertex is laveled with the value $V_2(x)$.
  The leaves are 0 by definition, though we study below also other assignments of leaf values.}
  \label{fig:tree_eg}
\end{figure}

Our main results address the tightness and convergence of $V_n$ as $n\to\infty$.

\begin{thm}\label{thm:tightness}
  Fix $d \ge 2$. Then $\{V_n\}_{n=1}^\infty$ is tight.
  Moreover, $\sup_n \Pr(|V_n| \ge 2k) \le e^{-cd^k}$ for some $c>0$ and all sufficiently large $k$.
\end{thm}

\begin{thm}\label{thm:distributional-convergence}
  Fix $d \ge 3$. Then $V_n$ converges in distribution as $n\to\infty$.
\end{thm}

\begin{thm}\label{thm:almost-sure-convergence}
  Fix $d \ge 15$. Then $V_n$ converges almost surely as $n\to\infty$.
\end{thm}

We believe that almost-sure convergence holds for all $d\ge 3$.
For $d=2$, we are unable to prove convergence in distribution, but numerical computations give strong indications that this should be true.
Curiously, there seem to be important qualitative differences between the behavior for $d=2$ and $d\ge 3$.
For $d=2$ there is strong dependence on boundary conditions, which does not occur for $d\ge 3$.
For example, simulations show the following: if one places on the leaves values of 0 or 2 according to independent Bernoulli random variables of parameter $p$, then the limiting law of the root depends continuously on $p$ for $d=2$.
In contrast, for $d\ge 3$, it ``jumps'' from one limit to another as $p$ goes from 0 to 1 (possibly passing through a third limit at a critical point).
While we are unable to fully establish this picture, we provide some results in this direction.

For a vertex $x\in T_d$, let $V_n(x)$ denote the outcome of the game if the game begins at vertex $x$, and terminates as before at level $2n$.
Note that for $x$ at an odd level the minimizing player will make the first move.
If $x$ is at level $2n$ or higher, we have by convention $V_n(x)=0$.
It is easy to see that at each level of the tree the values $V_n(\cdot)$ are i.i.d. with some distribution supported on even or edd integers according to the parity of the level.
Moreover, there is a recursion relating the distribution at a level to the distribution of the values at the next level.
It will be convenient to work with the cumulative distribution functions.
Let $\Phi$ (resp.\ $\Phi'$) be the map describing the evolution of a CDF at a single ``max'' (resp. ``min'') step.
Let $\Psi = \Phi \circ \Phi'$ describe the effect of a full round of the game.
Writing in terms of CDFs, we have
\[
  \Phi F = (\mu*F)^d  \qquad \Phi' F = 1-(1-\mu*F)^d.
\]
where $\mu$ is the PDF of the edge cookies $X$, and $\mu*F$ is the usual convolution.

In particular, if $F_0$ is the CDF corresponding to the constant value 0, then $\Psi^n F_0$ is the CDF corresponding to $V_n$, the value of the $n$-round game.
If $V_n$ converges in distribution, then the CDF of the limit $V_\infty$ must be a fixed point of $\Psi$.
Thus it is of interest to understand the different fixed points of $\Psi$.
Let $S$ be the operator corresponding to the map $V \mapsto V-1$ at the level of random variables (i.e., $(SF)(i)=F(i+1)$).
Let $R=S$ be the operator corresponding to $V \mapsto -(V+1)$ (i.e., $(RF)(i)=1-F(-i-2)$, since all distributions are supported on integers).

\begin{thm}\label{thm:fixed-points}
  The following are true:
  \begin{enumerate}[nosep,label=(\roman*)]
  \item For $d=2$, there exists a continuous one-parameter family $(F_\alpha)_{\alpha \in (0,1)}$ of CDFs which are fixed points for $\Psi$ and satisfy that $F_\alpha(0)=\alpha$ and $\Phi' F_\alpha =SF_\alpha$ (equivalently, $\Phi F_\alpha = S^{-1}F_\alpha$).
    Moreover, for any $\alpha \in (0,1)$, $F_\alpha$ is the unique such CDF.
  \item For $d \ge 3$, there is no CDF $F$ which is a fixed point for $\Psi$ and satisfies $\Phi' F =SF$.
  \item For $d \ge 15$, there is a CDF $F$ which is a fixed point for $\Psi$ and satisfies $\Phi' F =RF$.
    Indeed, the limiting law of $V_n$ is such. For $d \gg 1$, this fixed point is stable.
  \end{enumerate}
\end{thm}

We do not give a precise definition of a stable fixed point here.
Suffice it to say that if one starts from an initial CDF which is a slight perturbation of $F_\infty$ (say in the sup norm), then the obtained sequence still converges to $F_\infty$.
In particular, stable fixed points are isolated.
See \cref{rem:stability} and the results in \cref{sec:stability} for precise statements of stronger properties.
This highlights the different nature of the game for $d=2$ and large $d$. For $d=2$, there is a continuous family of fixed points, whereas for $d$ large, there are isolated fixed points.
In addition, the symmetries of the fixed points are also very different.

We end with some asymptotic results as $d$ tends to infinity.
Our methods allow one to derive additional terms in an asymptotic expansion for the distribution of $V_n$ as $d\to\infty$, though we only give the first term for values $\{-2,0,2\}$ below.

\begin{thm}\label{thm:asymptotics}
    We have as $d \to \infty$,
    \begin{align*}
      \mathbb P(V_\infty=0) &= 1 - (\tfrac d2 + 1)2^{-d} &\pm O(d^3 2^{-2d}) ,\\
     \mathbb P(V_\infty=-2) &= 2^{-d} &\pm O(d^3 2^{-2d}) ,\\
     \Pr(V_\infty=2) &= \tfrac d2 2^{-d} &\pm O(d^3 2^{-2d}) ,\\
     \E V_\infty &= (d-2)2^{-d} &\pm O(d^3 2^{-2d}) .
    \end{align*}
\end{thm}

We refer the reader to a some visualization of various fixed points of $\Psi$ computed numerically from different stating points for the binary and ternary tree, highlighting the qualitative differences between them.
These are available at \href{https://www.math.ubc.ca/~angel/minmax_game}{\url{https://www.math.ubc.ca/~angel/minmax\_game}},
together with some of the code discussed below.

\begin{table}
    \centering
    \begin{tabular}{|c|ccccccc|}
    \hline
    d&-6&-4&-2&0&2&4&6  \\
  \hline
   2 & 0.000026 & 0.010220 &  0.181957 & 0.492413 & 0.285597 & 0.029560 & 0.000225 \\ 
   3 & $6.58 \cdot 10^{-10}$ & 0.001160 & 0.147582 & 0.642837 & 0.206682 &  0.001739 &  $9.87 \cdot 10^{-10}$ \\ 
   4 & $1.83 \cdot 10^{-18}$ & 0.000037 & 0.082769 & 0.761521 & 0.155599 & 0.000074 & $3.66 \cdot 10^{-18}$ \\ 
   5 & $7.68 \cdot 10^{-33}$ & $3.02 \cdot 10^{-7}$ & 0.041409 & 0.859262 & 0.099326 & $7.55 \cdot 10^{-07}$ & $1.92 \cdot 10^{-32}$ \\ 
   10 & $2.52 \cdot 10^{-256}$ & $1.10 \cdot 10^{-26}$ & 0.001017 & 0.993910 & 0.005073 & $5.51 \cdot 10^{-26}$ & $1.26 \cdot 10^{-255}$ \\ \hline
 \end{tabular}
 \caption{The pmf of the limiting distribution of $V_n$.}
 \label{table:fixed_point_pmf}
\end{table}

\subsection{Background}

There have been a handful of previous works on min-max models of similar flavor.
Pearl originally studied the min-max game on a regular tree \cite{pearl1980asymptotic}, where one puts i.i.d.\ Uniform $[0,1]$ cookies on the leaves, with no additional cookies elsewhere.
The game then proceeds exactly as in our setup, where players alternately try to maximize and minimize the cookie they get when they reach the leaves (the analogue in our version is that all cookies are 0 except on the edges adjacent to the leaves).
Alternatively, one may view our game as a modification of Pearl's game, where the values on the leaves are not i.i.d., but generated by a branching random walk.
Pearl proved that the value of the game converges (with no scaling) to a deterministic constant, which has an explicit characterization, and that the same is true when the cookies have an arbitrary common distribution whose CDF is continuous and strictly increasing.
It was subsequently shown in \cite{khan2005limit} that after a proper centering and rescaling of the value, one has convergence in distribution to a truly random value, whose distribution does not depend on the cookie distribution.

Later, a similar game was considered on Bienayme--Galton--Watson trees by Martin and Stasi\'nski \cite{martin2020minimax}, and by Holroyd and Martin \cite{holroyd2021galton}.
They showed that the (unscaled) limit could be truly random, depending on the offspring distribution. They too obtained results for the rescaled limit.
We point out that in the non-noisy version of the problem, the convergence of the distributional recursion of the CDF is much more straightforward to obtain because of inherent monotonicity (see, for e.g., the proof of \cite[Theorem 1.2]{martin2020minimax}).

In contrast, there is no obvious monotonicity in the distributional recursion of the noisy version of the model (see \eqref{eq:recursion_distribution}).
Thus our proofs of convergence (\Cref{thm:almost-sure-convergence,thm:distributional-convergence}) are significantly more involved and take up the majority of this article.
Indeed, even the tightness of $V_n$ is not obvious in our case, whereas it is obvious in the non-noisy version of the game.
There is also no clear characterization of the limit in terms of the cookie distribution for us, whereas such characterizations are some of the main results in \cite{pearl1980asymptotic,martin2020minimax}. 

Our model was studied by Devroye and Kamoun \cite{DK}.
The considered the $d$-ary tree where each edge is assigned an i.i.d. variable $X_e$ which could take other distributions.
Among their results is a law of large numbers for $V_n$, namely that under mild conditions, $n^{-1} V_n$ converges to a deterministic limit.
The limit value depends on the law of $X$ in a non-trivial way, though in our setting of uniform $\pm1$ weights it is easy to show the limit must be 0, so that $V_n=o(n)$.

\medskip

An important aspect of such models is the concept of \textbf{endogeny},
introduced by Aldous and Bandopadhyay \cite{AA_max} in the setting of a recursive distributional equation (RDE) and a recursive tree process (RTP).
We briefly describe the idea.
A RDE is an equation of the form
\[ Z \stackrel{d}{=} g(\xi_i,(Z_i)_{1 \le i \le N}).\]
Here, $N$ is a natural number (possibly random), $\xi$ is a random variable,
$(Z_i)_i$ are i.i.d.\ random variables with distribution $Z$ and independent of $(\xi_i)$, and $g$ is an explicit function with appropriate domain and range.
A RTP is a realization of such an RDE as a family of random variables $(Z_v,\xi_v)_{v \in \T_N}$ on a common probability space such that
\[ Z_v = g(\xi_v,(Z_u)_{u\text{ is a child of }v}) \qquad\text{for all $v$ almost surely} ,\]
with $(\xi_v)_v$ being i.i.d.\ copies of $(\xi_i)_{i=1}^N$,
and $(Z_v)_{v\text{ in level }n}$ being i.i.d.\ random variables with some fixed distribution $\mu_n$.
Moreover, we require the $Z_v$ in level $n$ to be independent even when conditioned on $(\xi_v)_{v\text{ in levels below }n}$.
The RTP is said to be invariant if $\mu_n$ is the same for all $n$.
Finally, an invariant RTP is said to be \textbf{endogenous} if the value $Z_\rho$ at the root $\rho$ of the tree is a measurable function of the i.i.d.\ random variables $(\xi_v)$.
See \cite[Sections 2.3 and 2.4]{AA_max} for details.

Returning to our context, we have an RDE of the form
$$
V_v \stackrel{d}{=} g((X_e)_{e \in \cS_v}, (V_u)_{u \in \cS'_v},),
$$
where $\cS_v$ contains the $d(d+1)$ edges in the subtree of depth 2 below $v$, and $\cS'_v$ contains the $d^2$ grandchildren of $v$. We recall that $X_e$ is the cookie attached to an edge $e$.
Here, $V_v$ should be thought of as the value of the game played in the subtree rooted at $v$, and $g$ is the corresponding min-max function. Note that we only consider here vertices that are at an even distance from the root.
To each CDF $F$ that is a fixed point for $\Psi$, corresponds a unique invariant RTP given by the same equation above, but where the equality now holds almost surely, with the $V_v$ distributed according to $F$.
When $V_n$ converges in distribution, the CDF $F_\infty$ of the limiting distribution is such a fixed point.
When $V_n$ converges almost surely, the invariant RTP corresponding to $F_\infty$ can be constructed explicitly by taking the appropriate almost sure limit at each vertex.
In particular, the value at the root is a function solely of the cookies, which means that this RTP is endogenous.
In general, if $V_n$ converges in distribution, then the corresponding invariant RTP is endogenous if and only if $V_n$ converges almost surely.
In particular, this leads us to the following simple consequence of \Cref{thm:almost-sure-convergence}.

\begin{corollary}
  The invariant recursive tree process corresponding to the limiting distribution of the noisy min-max game (with zero boundary conditions) is endogenous for $d \ge 15$.
\end{corollary}

\subsection{Notation and formal definitions.}\label{sec:def}

Recall that every edge $e$ in $\T_d$ has an associated cookie $X_e$, with $(X_e)$ being independent random variables which are uniform on $\{-1,1\}$.
As a matter of convenience, for a site $v \neq \rho$, we write $X_v$ for $X_e$, where $e$ is the edge between $v$ and its parent.
We write $N(v)$ for the set of children of $v$.

Fix $n \ge 0$. Recall that we have defined the value $V_n$ of the $n$-round game. This can be thought of as the ``final'' value of the $n$-round game. Let us extend this definition to ``intermediate'' values $V_n(v)$ for all $v$ of height at most $2n$.
We first set the boundary values to be $V_n(v)=0$ for all vertices $v$ at height $2n$.
Next, given a vertex $v \in \T_d$ at height less than $2n$, we define
\begin{equation}\label{eq:game-recursive-def}
\begin{aligned}
V_n(v) &= \max\{ X_{uv}+V_n(u) : u \in N(v) \}, & v\text{ even},\\
V_n(v) &= \min\{ X_{uv}+V_n(u) : u \in N(v) \}, &v\text{ odd} .
\end{aligned}
\end{equation}
Observe that $V_n(\rho)=V_n$.
We may regard $V_n(v)$ as the value of the game played on the subtree rooted at $v$ up to depth $2n$, with the maximizer (resp.\ minimizer) playing first if $v$ is even (resp.\ odd).
Note also that $V_n(v)$ is even for $v$ even, and odd for $v$ odd.
Observe also that the distribution of $V_n(v)$ depends only on $2n-|v|$, where $|v|$ is the height on $v$.

We say that a CDF $F$ is even if it is the CDF of a random variable supported on even integers (equivalently, if $F(2i)=F(2i+1)$ for all $i \in \Z$).
Recall that $\Phi$ and $\Phi'$ are the maps describing the evolution of a CDF at a single ``max'' or ``min'' step of the game, respectively, and that $\Psi = \Phi \circ \Phi'$ describes the evolution of one round of the game. In particular, letting $F_{2n}$ be the CDF of $V_n$, and $F_{2n-1}$ the CDF of $V_n(u)$ for some child $u$ of $\rho$, we have that
\begin{equation}\label{eq:recursion_distribution}
 F_{n+1} = 
 \begin{cases}
    \Phi'(F_n) &\text{ if }n \text{ is even}\\
    \Phi(F_n) &\text{ if }n \text{ is odd},
 \end{cases}
\end{equation}
and hence, $F_{2n+2}=\Psi(F_{2n})$.

The above notations make sense for any tree and any distribution of the cookies.
In our case of the $d$-ary tree with uniform $\{\pm1\}$ cookies,
$\Phi$ and $\Phi'$ are the maps from $[0,1]^\Z$ to $[0,1]^\Z$ given by
\begin{equation}\label{eq:A-Phi}
  (\Phi(x))_i := \left(\frac{x_{i-1}+x_{i+1}}{2}\right)^d \qquad\text{and}\qquad (\Phi'(x))_i := 1-\left(1-\frac{x_{i-1}+x_{i+1}}{2}\right)^d.
\end{equation}
(More generally, if the cookie distribution has PDF $p$, then $\Phi = (p*x)^d$ and $\Phi'=1-(1-p*x)^d$ where $p*x$ denotes convolution.)
From here we also see that 
\begin{equation}\label{eq:Psi-g}
(\Psi(F))(i) = g\big(F(i-2),F(i), F(i+2)\big) ,
\end{equation}
where $g:[0,1]^3 \to [0,1]$ is defined by
\[ g(x,y,z) := \left[1-\frac12 \left(1-\frac{x+y}{2}\right)^d  - \frac12\left(1-\frac{y+z}{2}\right)^d\right]^d .\]

\subsection{Outline of the proofs}
The proofs involve a mix of percolation type arguments for large $d$ and contraction type arguments for small $d$. 

The proof of \Cref{thm:fixed-points} for $d=2$ involve an analysis of the iteration along with symmetry arguments. The claimed symmetry of the fixed point distribution for $d=2$ was suggested by simulations. Once we combine the symmetry of the fixed point CDF with the fact that it is a fixed point of the iteration, it is fairly quick to see that the fixed point CDF at $i$ can be written as an explicit function of the fixed point CDF at $i +2$. Once reduced to this, after some calculus, one can see that there is no valid solution to the equations for $d\ge 3$, and there is a continuum of solutions for $d=2$. The symmetry result for $d\ge 15$ can be seen by comparing the distribution the values between trees of even and odd heights, assuming we have almost sure convergence. 

The proofs of \Cref{thm:tightness} for $d=2$ and double exponential tail involve analysis of the CDF. The tightness for $d=2$ involve a different certain symmetry of the CDF (valid for all finite trees) along with further analysis of the  CDF iterations.

The proof of almost sure convergence (\Cref{thm:almost-sure-convergence}) involves a percolation type argument. The idea is that if $d $ is large, then we expect the value at even heights to be $0 $ and at odd heights to be $-1$ with very high probability. This can be seen fairly easily near the leaves by analyzing the minimum and maximum operations. This motivates defining a `defect' to be the set of odd vertices with height not equal to $-1$ and the set of even vertices with height not equal to $0$. Then a Peierl type argument is employed to show that the defect set is small and has an exponential tail. One difficulty in executing this idea to get almost sure convergence is that the values change as the height of the tree changes. So we need to work with the union of the defects for all heights $n$. For $d \ge 15$, it turns out that even the union of the defects is small. We did not try to optimize the value of $d$, although we doubt that this can be made to go all the way down to $d=3$.

The proof of the distributional convergence result (\Cref{thm:distributional-convergence}) involves a Banach fixed point type argument along with computer assistance in order to verify a certain ``initial condition'' (but the proof is rigorous). The idea is as follows. Let $F_{2n}$ be the CDF of $V_{n}$. We want to show that the operator $\Psi$ which maps $F_{2n}$ to $F_{2n+2}$ (see \eqref{eq:DPsi}) is a contraction. However, it cannot be a contraction on the whole space because of the existence multiple fixed points. Nevertheless, if we iterate the operator enough times, it turns out that the CDF enters a `small' convex set $U$ containing the limiting fixed point where the operator is a contraction. Once we show that we enter $U$ and never leave, we need to estimate the supremum over all $x \in U$ of the operator norm of the derivative $D\Psi$ evaluated at $x$ (if we show this quantity is $<1$, we are done by a simple contraction argument). It turns out that the operator norm is smaller than 1 if $d \ge 5$, but is still larger than 1 if $d=3,4$. Nevertheless, it is the case that the largest eigenvalue of the operator is smaller than 1. To utilize this, we need to conjugate $\Psi$ by the suitable linear operator $E$, which pushes the operator norm of derivative of the conjugated operator close to its largest eigenvalue. 

We are left to argue that the existece of the small convex set $U$. This is done in the computer employing \emph{interval arithmetic} in Python. The idea is that rather using floating point computations which are non-precise and would therefore accumulate errors, we instead ask the computer to record an interval which is \emph{guaranteed} to contain the correct value of the floating point. Then, instead of doing arithmetic with numbers, the computer does arithmetic with intervals, in each step making sure that the correct value is contained (rigorously) within the interval it outputs. We have added \Cref{sec:interval_arithmetic} for an exposition of the necessary background on this concept.

Since we are dealing with iteration of CDFs with potentially infinite support, we first need to reduce the iteration to a finite dimensional problem. It turns out that tracking the CDF $F_{2n}(i)$ for $i\in \{-4,-2,0,2,4\}$ is enough (this is not so surprising in light of the fast double exponential decay). The CDF on these 5 points are tracked using interval arithmetic.

\paragraph{Acknowledgement:}

This problem was brought to our attention by Xingshi Cai at a Discrete Probability workshop at Mcgill's Bellairs institute.
The tightness argument in \Cref{sec:tight_d_2} is due to Svante Janson.
The authors would also like to thank Kesav Krishnan for his contributions during the initial stages of this project, and to Svante Janson for useful discussions.

During this work OA was supported by NSERC, Clay senior fellowship at SLMath and a visiting fellowship at Magdalen college, Oxford.
The research of GR is supported by NSERC 50311 57400.
The research of YS is supported in part by ISF grant 1361/22.

\section{A continuum of fixed points for $d=2$}

Our goal in this section is to prove items (i) and (ii) of \Cref{thm:fixed-points}.
Recall the definition of the shift operator $(SF)(i) = F(i+1)$ and
the operations $\Phi,\Phi'$ from \eqref{eq:A-Phi}, and that $\Psi=\Phi \circ \Phi'$.
Note that all of these commute with $S$.

\begin{proof}[Proof of \cref{thm:fixed-points}(i)]
  Fix $d=2$.
  We first observe that a CDF $F$ has $\Phi(F)=S^{-1}F$ if and only if it satisfies $\Phi'(F)=SF$. Indeed, 
  \begin{align*}
    (\Phi'F)(i) = ((A\Phi A)(F))(i) &= 1 - (\Phi(\1-F))(i) = 1 - \left(1-\frac{F(i-1)+F(i+1)}2\right)^2 \\&= F(i-1)+F(i+1) - \left(\frac{F(i-1)+F(i+1)}2\right)^2 \\&= F(i-1)+F(i+1) - (\Phi(F))(i) .
  \end{align*}
  Thus, $\Phi(F) + \Phi'(F) = SF + S^{-1}F$, proving the claimed equivalence.
  Let us also note that these equivalent conditions imply that $\Psi(F)=F$.
  Indeed,
  \[ \Psi(F) = (\Phi \circ\Phi')(F) = \Phi(SF) = S \Phi(F) = S S^{-1} F = F .\]
  In summary, if $F$ is a fixed point of $S\Phi$, then it is also a fixed point of $\Psi$.
  
  Fix $\alpha \in (0,1)$.
  We wish to show that there exists a unique CDF $F_\alpha$ satisfying
  \[ F_\alpha(0) = \alpha \qquad\text{and}\qquad \Phi (F_\alpha) = S^{-1} F_\alpha .\]
  The second condition can be re-written to say that for all $i$,
  \[ F_\alpha(i-2) = (S^{-1}F_\alpha)(i-1) = (\Phi(F_\alpha))(i-1) = \tfrac14 (F_\alpha(i-2)+F_\alpha(i))^2 .\]
  Equivalently, 
  \begin{equation}
    F_\alpha(i) = h(F_\alpha(i-2)) \qquad \text{where} \qquad
    h(x) := 2\sqrt{x}-x. \label{eq:forward_iteration}
  \end{equation}
  Note that $h \colon [0,1] \to [0,1]$ is strictly increasing on $[0,1]$, has fixed points $0$ and $1$, and has $h(x)\geq x$.
  Therefore, once we fix $F_\alpha(0)=\alpha$, the values of $F_\alpha(2i)$ for all $i \in \Z$ are uniquely determined to be $F_\alpha(2i) = h^{(i)}(\alpha)$, and they yield a valid CDF.
  Continuity in $\alpha $ is also clear from this description.
\end{proof}

\begin{remark}
The above proof does not rule out the existence of fixed points of $\Psi$ which are not simultaneously also fixed points of $S\Phi$ and $S^{-1}\Phi'$.
Numerically we observe that for any distribution chosen for even values at the leaves, the law of $V_n$ converges to one of the distributions $F_\alpha$ above.
If the values at the leaves are i.i.d.\ integers of arbitrary parity, the parity of $V_n$ is also random.
Numerically, it is always one of the CDF's $\{F_{\alpha,\beta}\}$, which are determined by \eqref{eq:forward_iteration} and $F_{\alpha,\beta}(0)=\alpha, F_{\alpha,\beta}(1)=\beta$.
This determines a fixed point of $\Psi$ for any $\alpha\in(0,1)$ and $\beta\in[\alpha,h(\alpha)]$.
\end{remark}

\begin{proof}[Proof of \cref{thm:fixed-points}(ii)]
  Fix $d \ge 3$ and assume toward a contradiction that $F$ is a fixed point for $\Psi$ such that $\Phi'(F)=SF$.
  Note that $$F=\Psi(F)=(\Phi \circ \Phi')(F) = \Phi(SF)=S \Phi(F),$$ which implies that $\Phi(F)=S^{-1}F$.
  A computation similar to the derivation of \eqref{eq:forward_iteration} yields that $F$ satisfies the recursion $F(i)=h_d(F(i-2))$, where $h_d$ is given by $x \mapsto 2x^{1/d}-x$.

  Define $\tilde F(i) = 1 - F(i)$.
  Using now $\Phi'(F)=SF$, which is equivalent to $\Phi (\tilde F)=S(\tilde F)$,
  we see that $\tilde F$ satisfies the similar recursion (in the opposite direction): 
  $\tilde F(i-2)=h_d(\tilde F(i))$.
  Noting that for $d>2$ we have  $h'(1)<0$, 
  we see that $h$ takes values strictly larger than $1$ for all values of $x$ close to 1.
  Since $F(i) \le 1$ for all~$i$, and $F(i) \to 1$ as $i \to \infty$, the first recursion implies that $F(i)=1$ for some $i$.
  The second recursion then implies that $\tilde F(j)=0$ (i.e., $F(j)=1$) for all $j \le i$, which is impossible.
\end{proof}

\section{Tightness and doubly exponentially decaying tails}

In this section, we prove \cref{thm:tightness}.
This is a combination of several arguments.
Tightness of $\{V_n\}$ when $d=2$ is proved directly.
Tightness for $d\ge3$ is a consequence of the convergence in distribution shown in \cref{thm:distributional-convergence} (the proof of which in \cref{sec:convergence-in-law} does not rely on the apriori knowledge of tightness).
Finally, we show that tightness implies the double exponential tail for any $d \ge 2$.
Our proofs here are written for the case of 0 values on the leaves, but hold for more general bounded leaf values with minimal changes.
The proof for $d=2$ is quite specific to the case of uniform $\pm1$ cookies.

Recall from \cref{sec:def} that $F_{2n}$ is the CDF of $V_n$ and that $F_{2n+2}=\Psi(F_{2n})$. Recall also from \eqref{eq:Psi-g} that
\[ (\Psi(F))(i) = g(F(i-2),F(i), F(i+2)), \]
where $g:[0,1]^3 \to [0,1]$ is given by
\begin{equation}\label{eq:g}
 g(x,y,z) := \left[1-\frac12 \left(1-\frac{x+y}{2}\right)^d  - \frac12\left(1-\frac{y+z}{2}\right)^d\right]^d .
\end{equation}
This relates the CDF of $V_{n+1}$ to that of $V_n$, which describes the lower tails in the sense that
\begin{equation}\label{eq:g-P}
 \Pr(V_{n+1} \le k) = g(\Pr(V_n \le k-2),\Pr(V_n \le k),\Pr(V_n \le k+2)) .
\end{equation}
For the upper tails, defining
\begin{equation}\label{eq:h}
h(x,y,z) := 1-g(1-x,1-y,1-z) = 1 - \left[1-\frac12 \left(\frac{x+y}{2}\right)^d  - \frac12\left(\frac{y+z}{2}\right)^d\right]^d ,
\end{equation}
we similarly have that
\begin{equation}\label{eq:h-P}
 \Pr(V_{n+1} \ge k) = h(\Pr(V_n \ge k-2),\Pr(V_n \ge k),\Pr(V_n \ge k+2)) .
\end{equation}

\subsection{Double exponential tails}

We first show that tightness implies double exponential tails for any $d \ge 2$.

\begin{prop}\label{prop:double-exp-tails}
    Fix $d \ge 2$. For $k \in \Z$, denote $p_k^+ := \sup_n \Pr(V_n \ge 2k)$ and $p_k^- := \sup_n \Pr(V_n \le 2k)$. Then for any integer $m$ and positive integer $k$,
    \[ p^+_{m+k} \le \tfrac12 \big(2p^+_m\big)^{d^k} \qquad\text{and}\qquad p^-_{m-k} \le \tfrac1{2d} \big(2d p^-_m\big)^{d^k} .\]
    In particular, if $\{V_n\}$ is tight, then $\sup_n \Pr(|V_n| \ge 2k) \le e^{-cd^k}$ for some $c>0$ and all $k$ large enough.
\end{prop}
  
\begin{proof}
  We start with the upper tail.
  Recall the recursion~\eqref{eq:h} and~\eqref{eq:h-P} for $1-F_n$.
  Since $h$ is non-decreasing in each of the coordinates, we have
  \[ p^+_{k+1} \le h(p^+_k,p^+_{k+1},p^+_{k+2}) \le h(p^+_k,p^+_k,p^+_k). \]
  Observe that
  \[
    h(x,x,x) = 1-(1-x^d)^d \leq d x^d, \]
  for $x\in[0,1]$.
  Denoting $q_k := 2p^+_k$, we get that $q_{k+1} = 2p^+_{k+1} \le 2d (p^+_k)^d = 2d 2^{-d} q_k^d \le q_k^d$.
  Hence, $q_{m+k} \le q_m^{d^k}$, so that $p^+_{m+k} \le \tfrac12 (2p^+_m)^{d^k}$.
  
  We now bound the lower tail in a similar way.
  Recall the recursion~\eqref{eq:g} and~\eqref{eq:g-P} for $F_n$.
  We have
  \[ p^-_{k-1} \le g(p^-_{k-2},p^-_{k-1},p^-_k) \le g(p^-_k,p^-_k,p^-_k). \]
  Similarly, we have
  \[
    g(x,x,x) = \big(1 - (1-x)^d \big)^d \le \big(1-(1-dx)\big)^d = (dx)^d .\]
  Denoting $q_k := 2d p^-_k$, we get that $q_{k-1} = 2d p^-_{k-1} \le 2d(dp^-_k)^d \le 2d 2^{-d} q_k^d \le q_k^d$.
  Hence, $q_{m-k} \le q_m^{d^k}$, so that $p^-_{m-k} \le \tfrac1{2d} (2dp_m)^{d^k}$.

  Finally, if $\{V_n\}$ is tight, then $p^+_m \to 0$ as $m \to \infty$, and $p^-_m \to 0$ as $m \to -\infty$, so that we make choose $m$ large enough so that $2p^+_m \le \frac1e$ and $2dp^-_m \le \frac1e$, which gives that
  \[ \sup_n \Pr(|V_n| \ge 2k) \le p^+_k + p^-_k = p^+_{m+(k-m)} + p^-_{-m-(k-m)} \le \tfrac12 e^{-d^{k-m}} + \tfrac1{2d} e^{-d^{k-m}} \le e^{-d^{k-m}} = e^{-cd^k} . \qedhere
  \]
\end{proof}

\subsection{Tightness for $d=2$}\label{sec:tight_d_2}

The following is a variation on a proof due to Svante Janson.
The proof relies on a certain relation between the distributions of $V_n=V_n(\rho)$ and $\bar V_n:=V_n(u)$, where $u$ is a fixed child of $\rho$.
This relation is specific to the case of $\pm$ cookies with 0 boundary values.

\begin{lemma}\label{lem:symmetry-relation}
  Fix $n \ge 1$. Then $V_n \stackrel{d}{=} -\bar V_n + \xi_n$ for some random variable $\xi_n$ (which may not be independent of $V_n$) having $\xi_n \in \{+1,-1\}$.
\end{lemma}

\begin{proof}
  The values $V_n(v)$ have been defined in~\eqref{eq:game-recursive-def} so as to represent the ``intermediate'' values in the $n$-round game. Recall that in this game, players alternate turns, until they reach height $2n$ of the tree, where the boundary values are all 0.
  If the game begins at the root, there are $n$ rounds with the maximizer moving first.
  Let $V'_n(v)$ be defined in a similar manner (with respect to the same cookies), but the game continues until level $2n+1$ of the tree, and the  boundary condition is taken to be $V'_n(w)=0$ for all $w$ at height $2n+1$.
  The values $V'_n(v)$ represent the intermediate values for the same game.
  If we begin at the root, the first player (who is still the maximizer) now has $n+1$ turns, while the second player still has $n$ turns.
    
    Fix a child $u$ of $\rho$. {It is straightforward to see that $V'_n(u) = -V_n(\rho)$ in distribution. Indeed, consider flipping the signs of all the cookies. Let $V_n''(u)$ be the value of the in the subtree (of height $2n$) below $u$ where we reverse the roles of max and min. The value $V'_n(u)$ for the original cookies is exactly equal to $-V''_n(u)$ for the flipped cookies (this can be seen easily by induction, starting from the leaves and going towards the root). Now note that $V_n''(u)=V_n(\rho)$ in distribution since the height of the tree below $u$ is $2n$ for $V_n''$ and the roles of max and min match for both. Since the cookies are uniform, $V'_n(u) = -V''_n(u) = -V_n(\rho)$ in distribution.}

    In addition, since $V_n(w)=0$ and $|V'_n(w)| \le 1$ for all $w$ at height $2n$, a moment's thought reveals that $V'_n(u)-1 \le V_n(u) \le V'_n(u)+1$ almost surely. Indeed, one can think of $V_n'(u)$ as solving the same optimization problem but with random values $V_n'(w)$ for all leaves $w$. The inequality is straightforward from this observation.
    This completes the proof.
\end{proof}

\begin{corollary}\label{cor:skew-symmetry}
    For $d=2$, we have $\Pr(V_n \ge 2) \le \alpha$ for all $n$, where $\alpha=\frac12(\sqrt5-1)$ is the golden mean.
\end{corollary}

\begin{proof}
    By the previous lemma, $\Pr(V_n \ge 2) = \Pr(-\bar V_n+\xi_n \ge 2) \le \Pr(-\bar V_n \ge 1)$. On the other hand, $\Pr(V_n \le 0) \ge \Pr(\bar V_n \le -1)^2$. Thus,
    \[ \Pr(V_n \ge 2)^2 \le \Pr(\bar V_n \le -1)^2 \le \Pr(V_n \le 0) = 1- \Pr(V_n \ge 2) .\]
    Since $x^2 \le 1-x$ holds (for $x \ge 0$) if and only if $x \le \alpha$, we deduce that $\Pr(V_n \ge 2) \le \alpha$.
\end{proof}

\begin{lemma}\label{lem:skew-symmetry}
Fix $n \ge 1$. Then $V_n \ge_{st} -V_n-2$.
\end{lemma}

\begin{proof}
    Recall $V_n'$ from the proof of \Cref{lem:symmetry-relation}. Fix a child $u$ of $\rho$.
    We have seen that $V_n(u) \ge V'_n(u)-1$ almost surely and that $V'_n(u)=-V_n(\rho) $ in distribution. Clearly, $V_n = V_n(\rho) \ge V_n(u)-1$ almost surely. In particular, $V_n \ge_{st} -V_n-2$.
\end{proof}

We are now ready to establish tightness.

\begin{prop}\label{prop:tightness}
    For $d=2$, the family of random variables $\{V_n\}_{n=0}^\infty$ is tight.
\end{prop}

\begin{proof}
  It suffices to provide a bound on the upper tail that tends to zero uniformly in $n$, that is, $\sup_n \Pr(Z_n \ge k) \to 0$ as $k \to \infty$.
  Indeed, the analogous fact for the lower tail then follows from \cref{lem:skew-symmetry}.
  We shall show that $\Pr(V_n \ge 2k) \le p_k$ for all $n \ge 0$ and $k \ge 1$, where $p_k := \alpha \cdot \beta^{k-1}$ with $\alpha=\frac{\sqrt5-1}{2}$ as before and $\beta:=1/\sqrt5$.

  We prove this by induction on $n \ge 0$.
  The induction base $n=0$ holds trivially.
  Let $n \ge 0$ and assume the assertion is true for $V_n$, that is, $\Pr(V_n \ge 2k) \le p_k$ for all $k \ge 1$.
  We need to show that $\Pr(V_{n+1} \ge 2k) \le p_k$ for all $k \ge 1$.
  We have already established this for $k=1$ in \cref{cor:skew-symmetry},
  so it remains to show this for all $k \ge 2$.
  Recalling~\eqref{eq:h} and~\eqref{eq:h-P}, we need to show is that for all $k \ge 2$,
  \[ h\big(\Pr(V_n \ge 2k-2), \Pr(V_n \ge 2k), \Pr(V_n \ge 2k+2)\big) \le p_k .\]
  Since $h$ is increasing in each of its variables, it suffices to show for all $k \ge 2$ that
  \[ h\big(p_{k-1},p_k,p_{k+1}\big) \le h\big(p_{k-1},p_k,p_k\big) \le p_k .\] 
  Since $p_k= \beta \cdot p_{k-1}$ and $p_{k-1} \le \alpha$ for $k \ge 2$, it suffices to show that for all $0 < x \le \alpha$,
  \[ h(x,\beta x,\beta x) = 1 - \big(1- \tfrac18(x+\beta x)^2 - \tfrac12(\beta x)^2\big)^2 \le \beta x .\]
  This expands to
  \[ \tfrac14(x+\beta x)^2 + (\beta x)^2 - \big(\tfrac18(x+\beta x)^2 + \tfrac12(\beta x)^2\big)^2 \le \beta x .\]
  Discarding the negative term on the left-hand side, dividing by $\beta$ and rearranging, it suffices to show that
  \[ \big(\tfrac 14(1+\beta)^2 + \beta^2 \big) x \le \beta .\]
  Finally, note that the coefficient precisely equals $\beta/\alpha$, so this indeed holds for all $x\leq\alpha$.
  This completes the proof.
\end{proof}

\section{Almost-sure convergence for large $d$}\label{sec:large_d}

In this section we prove \cref{thm:almost-sure-convergence} and \cref{thm:fixed-points}(iii).
The main idea behind our proof is that when $d$ is large, an even vertex typically has value $0$, while an odd vertex typically has value $-1$.
To capture this, let $B_n$ denote the ball of radius $n$ around the root, and define
\begin{equation}\label{eq:S+-_n}
  S^+_n := \{ v \in B_{2n}: V_n(v) > 0 \} \qquad\text{and}\qquad
  S^-_n := \{ v \in B_{2n}: V_n(v) < -1 \} .
\end{equation}
We think of $S^+_n$ (resp.\ $S^-_n$) as the set of vertices whose value in the $n$-round game is atypically high (resp.\ low).
The following claim describes some basic properties of these sets and their relation to the underlying cookies $X_e$.
Recall that $N(v)$ denotes the children of vertex $v$.

\begin{claim}\label{cl:boundary-values}
    Fix $n$ and $v \in B_{2n}$.
    \begin{enumerate}[nosep]
    \item If $v \in S^+_n$ is even, then $N(v) \cap S^+_n \neq \emptyset$.
    \item If $v \in S^-_n$ is odd, then $N(v) \cap S^-_n \neq \emptyset$.
    \item If $v \in S^+_n$ is odd, then $V_n(u)=0$ and $X_{uv}=1$ for all $u \in N(v) \setminus S^+_n$.
    \item If $v \in S^-_n$ is even, then $V_n(u)=-1$ and $X_{uv}=-1$ for all $u \in N(v) \setminus S^-_n$.        
    \end{enumerate}
\end{claim}

These claims are all simple combinatorial consequences of the definition of $S^\pm$ and the recursions for values.
For example if the value at an even (max) vertex is high then at least one of its children must have a high value.
At odd vertices all of its children must have high values.
A detailed proof is below for completeness.

\begin{proof}  
  Suppose $v \in S^+_n$ is even. Then $V_n(v)>0$ and it follows that $V_n(v) \ge 2$ since $V_n(v)$ is even. Hence, $\max\{u \in N(v) : V_n(u) \} \ge 1$. Thus, there exists $u \in N(v)$ such that $V_n(u) \ge 1$, so that $u \in S^+_n$ and $N(v) \cap S^+_n \neq \emptyset$.
  This proves the first clause.

 Suppose $v \in S^-_n$ is odd. Then $V_n(v) \le -3$ (since $V_n(v)$ is odd) so that $\min\{u \in N(v) : V_n(u) \} \le -2$. Thus, there exists $u \in N(v)$ such that $V_n(u) \le -2$, so that $u \in S^-_n$ and $N(v) \cap S^-_n \neq \emptyset$.

 Suppose that $v \in S^+_n$ is odd. Then $V_n(v) \ge 1$ so that $\min\{ u \in N(v) : X_{uv}+V_n(u) \} \ge 1$. Thus, $X_{uv}+V_n(u) \ge 1$ for all $u \in N(v)$. Let $u \in N(v) \setminus S^+_n$. Then $V_n(u)<1$ so that $V_n(u)=0$ (since $V_n(u)$ is even) and $X_{uv}=1$.
 
 Suppose that $v \in S^-_n$ is even. Then $V_n(v) \le -2$ so that $\max\{ u \in N(v) : X_{uv}+V_n(u) \} \le -2$. Thus, $X_{uv}+V_n(u) \le -2$ for all $u \in N(v)$. Let $u \in N(v) \setminus S^-_n$. Then $V_n(u) > -2$ so that $V_n(u)=-1$ (since $V_n(u)$ is odd) and $X_{uv}=-1$. 
\end{proof}

Let $S$ be the union of $S^+_n \cup S^-_n$ over all $n$.
Note that a vertex may be in $S_n^+$ for some $n$, in $S^-_m$ for some $m$, and in neither for other $n$.
Let $K$ be the connected component of the root in $S$.
We think of $S$ as the set of vertices which are atypical in any of the finite-round games.
The following observation shows that if these atypical vertices do not percolate, then the values of the $n$-round games eventually stabilize.
We then show in \cref{lem:K-prob-bound} below that these vertices indeed do not percolate when $d \ge 15$.

\begin{lemma}\label{lem:stabilization}
  If $K$ is contained in $B_{2N-2}$, then $V_n=V_N$ for all $n \ge N$.
\end{lemma}
  
\begin{proof}
    Note first that if $\rho \notin S$, then $V_n=V_n(\rho)=0$ for all $n$.
    Suppose that $\rho \in S$ and let $\partial K$ be the set of vertices not in $K$ that have at least one neighbor in $K$. We shall show that the values on $K \cup \partial K$ are the same when computed in $B_{2n}$ for any $n \ge N$. Let us first show this on $\partial K$. Let $v \in \partial K$ and note that this implies that $v \notin S$. If $v$ is even, then $v \notin S$ implies that $V_n(v)=0$ for all $n \ge N$, so that $V_N=V_n(v)=0$ does not depend on $n$. Similarly, if $v$ is odd, then $v \notin S$ implies that $V_n(v)=-1$, and again the value of $v$ does not depend on $n$. We now continue to show that the values on $K \cup \partial K$ are the same for all $n \ge N$. We prove this by reverse induction on the distance from the root. Note that the base of this induction has already been covered (as any site in $K \cup \partial K$ that is farthest away from the root is necessarily in $\partial K$). Let $v \in K \cup \partial K$ be arbitrary, and suppose by induction that we have already established that $V_n(u)=V_N$ for all $u \in K \cup \partial K$ that are (strictly) farther from the root than $v$ (in particular, for all the children of $v$ that are in $K \cup \partial K$). We may assume that $v \notin \partial K$, as we have already handled the complementary case. It follows from \eqref{eq:game-recursive-def} and the induction hypothesis (since all children of $v$ are in $K \cup \partial K$) that $V_n(v)$ does not depend on~$n$.
\end{proof}

\begin{lemma}\label{lem:K-prob-bound}
  Fix $N \ge 1$.
  Let $K_N$ be the connected component of the root in $(S^+_1 \cup S^-_1) \cup \cdots \cup (S^+_N \cup S^-_N)$. 
  For any connected set $A$ containing $\rho$ we have
  \[ \Pr(K_N = A) \le 3^k \cdot 2^{-\frac12(d-1)k},\]
  and
  \[ \Pr(|K_N| = k) \le (3ed)^k \cdot 2^{-\tfrac12 (d-1)k} .\]
\end{lemma}

\begin{proof}
  The second claim follows form the first, since there are at most $(ed)^k$ connected sets of size $k$ in $\T_d$ containing $\rho$ (see, e.g., \cite[Chapter~45]{bollobas2006art}).

  Denote $\cS^\pm_N := \bigcup_{n \le N} S^\pm_n$ (these two sets are not necessarily disjoint). Note that $K_N$ is the connected component of the root in $\cS_N=\cS^+_N \cup \cS^-_N$.
  Let $A^\pm$ be two sets such that $A = A^+ \cup A^-$. Since there are at most $3^k$ such decompositions, it suffices show that for any $(A^+,A^-)$,
  \[ \Pr(K_N \cap \cS^+_N = A^+\text{ and }K_N \cap \cS^-_N = A^-) \le 2^{-\frac12(d-1) |A|} .\]
  Suppose without loss of generality that this probability is non-zero.
    Let $\partial^* A := N(A) \setminus A$ denote the set of sites not in $A$ whose parent is in $A$. Then by the first two parts of \cref{cl:boundary-values} and the isoperimetric inequality given in \cref{lem:parity-iso**} below,
    \[ |\partial^* A^+ \cap \text{Even}| \ge \tfrac12(d-1)|A^+| \qquad\text{and}\qquad |\partial^* A^- \cap \text{Odd}| \ge \tfrac12(d-1)|A^-|. \]
    By the last two parts of \cref{cl:boundary-values}, on the event in question, every even $v \in \partial^* A^+$ has cookie $X_v=1$. Similarly, every odd $v \in \partial^* A^-$ has cookie $X_v=-1$.
    Thus,
    \begin{align*}    
     \Pr(K_N \cap \cS^+_N = A^+\text{ and }K_N \cap \cS^-_N = A^-)
      &\le \Pr(X\equiv 1\text{ on }\partial^* A^+ \cap \text{Even},~X\equiv -1\text{ on }\partial^* A^- \cap \text{Odd}) \\
      &= 2^{-|\partial^* A^+ \cap \text{Even}|-|\partial^* A^- \cap \text{Odd}|} \\
      &\le 2^{-\frac12(d-1) |A|}. \qedhere
     \end{align*}
\end{proof}

Note that $K_N$ increases to $K$. It thus suffices to prove the claimed bound (which does not depend on $N$) for $K_N$ instead of $K$. The advantage of $K_N$ over $K$ is that it is clearly finite (as it is contained in $B_{2N}$).

\begin{proof}[Proof of \cref{thm:almost-sure-convergence}]
  Let $d \ge 15$.
  Note that $K_N$ almost surely increases to $K$.
  Thus, using that $3ed \cdot 2^{-\frac12(d-1)}<1$, we obtain by \cref{lem:K-prob-bound},
  \[ \P(|K|=\infty) = \P(\lim_{N \to \infty} |K_N|=\infty) = \lim_{k \to \infty} \lim_{N \to \infty} \P(|K_N|>k) = 0 .\]
  Thus, $K$ is almost surely finite, and \Cref{lem:stabilization} now yields that $V_n$ almost surely stabilizes. This proves the almost-sure convergence of $V_n$.
\end{proof}

\subsection{A reflection symmetry of the fixed point}

Observe that if $V_n$ converges almost surely, then so does $V_n(v)$ for all $v \in \T_d$. Indeed, this holds for any even vertex $v$ of height $m$ due to the fact that $(V_{n-m}(v))_{n=0}^\infty$ has the same law as $(V_n)_{n=0}^\infty$, and it then follows for any odd vertex simply because the values of its children almost surely stabilize.
Recall the definitions of $\Phi,\Phi',R$ from the introduction.
The following establishes the first part of \cref{thm:fixed-points}(iii).

\begin{thm}\label{thm:R-symmetry}
   Fix $d \ge 15$. For $v \in \T_d$, let $V_\infty(v)$ be the almost-sure limit of $V_n(v)$ as $n \to \infty$.
   Then $V_\infty(\rho)=-V_\infty(u)-1$ in distribution, where $u$ is any fixed child of $\rho$. In particular, if $F$ and $F'$ are the CDFs of $V_\infty(\rho)$ and $V_\infty(u)$, respectively, then $F = \Phi F' = RF'$ and $F' = \Phi' F = RF$.
\end{thm}
 
\begin{proof}
    Let $V'_n(v)$ be defined as in~\eqref{eq:game-recursive-def}, but the boundary condition now taken to be $V'_n(v)=0$ for all $v$ at height $2n+1$. We have shown in the proof of \cref{lem:symmetry-relation} that if $u$ is a child of $\rho$, then $V'_n(u)=-V_n$ in distribution.
    On the other hand, we claim that $V'_n(u)$ converges almost surely, and that its limit is almost surely equal to $V_n(v)+1$. Indeed, this follows from the proof of \cref{thm:almost-sure-convergence}. Thus, $V_\infty(u)=-V_\infty(\rho)-1$ in distribution.
\end{proof}

\subsection{Stability}\label{sec:stability}

In this section we complete the proof of \cref{thm:fixed-points}(iii).
The following result establishes a large basin of attraction for the fixed point corresponding to $V_\infty$. 

\begin{thm}\label{thm:stability}
   Fix $d \ge 15$. Let $F_\infty$ be the CDF of $V_\infty$. There exists $\eps>0$ such that the following holds. Let $F$ be a (even) CDF having at atom of size at least $1-\eps$ at 0. Then $\Psi^n F \to F_\infty$ weakly as $n \to \infty$.
\end{thm}

Our proof of \cref{thm:stability} yields an explicit estimate on $\eps$.
For large $d$, it enables us to take $\eps = cd^{-2}$ for some constant $c>0$ (a posteriori, this can be bootstrapped to allow $\eps$ to be $1/d$ up to a logarithmic factor; in fact, if we restrict to ``one-sided'' CDFs, we can even take $\eps$ tending to one; see \cref{thm:bootstrap} below). Since the size of the atom of $F_\infty$ at 0 tends to 1 exponentially fast (via \Cref{lem:K-prob-bound}), this will imply, for large $d$, that $F_\infty$ has a large neighborhood (in sup norm) which lies entirely in the basin of attraction, implying, in particular, that $F_\infty$ is a stable fixed point (see \cref{rem:stability}).

In fact, we prove something much stronger than the weak convergence in \cref{thm:stability}.
Fix a configuration $\tau \in (2\Z)^{\T_d \cap \text{Even}}$. Given $n \ge 0$, define $V^\tau_n(v)$ as in~\eqref{eq:game-recursive-def}, but with boundary condition given by $V^\tau_n(v)=\tau(v)$ for every vertex $v$ at height $2n$. This corresponds to an $n$-round game, where the boundary condition is not all 0, but is instead given by the configuration $\tau$.
The following result shows that $V^\tau_n(\rho)$ converges to the same $V_\infty$ as from the usual game, uniformly for a large class of~$\tau$. The class we consider requires $\tau$ to be zero at most places, but allows it to take arbitrary values elsewhere.

\begin{thm}\label{thm:stability2}
   Fix $d \ge 15$. There exists $\eps>0$ such that the following holds.
   Let $(Y_v)$ be a collection of iid Bernoulli($\eps$) random variables, indexed by the even vertices of $\T_d$, and independent of the cookies $(X_e)$. Then, almost surely, there exists an $N$ such that $V^\tau_n(\rho)=V_N(\rho)$ for all $n \ge N$ and all $\tau$ such that $\tau(v)=0$ whenever $Y_v=0$.
\end{thm}

\begin{proof}
The proof is a modification of that of \cref{thm:almost-sure-convergence}.
Let $L_m$ denote the set of vertices at height~$m$. Let $\cT$ be the collection of all $\tau$ such that $\tau(v)=0$ whenever $Y_v=0$.
Define
\begin{equation}\label{eq:S+-_n2}
 S^+_n := \{v \in B_{2n} : \exists \tau \in \cT,~ V^\tau_n(v)>0\}, \qquad S^-_n := \{v \in B_{2n} : \exists \tau \in \cT,~ V^\tau_n(v)<-1\} .
\end{equation}
Note that these sets may now intersect $L_{2n}$ (in contrast to before), and that $v \in L_{2n}$ belongs to one of these sets precisely when $Y_v=1$. With this definition, all parts of \cref{cl:boundary-values} continue to hold (for $V^\tau_n(v)$) as long as the vertex $v$ under consideration is not in $L_{2n}$.

We slightly modify the definitions of $S$ and $K$ from before.
Fix $N_0 \ge 1$ and let $T$ be the set of vertices $v \in L_{2N_0} \cup L_{2N_0+2} \cup \cdots$ such that $Y_v=1$.
Let $S$ be the union of $S^+_n \cup S^-_n$ over all $n \ge N_0$, and let $K$ be the connected component of the root in $S \setminus T$. With these definitions, the proof of \cref{lem:stabilization} continues to hold (with minor changes) and yields the conclusion that $V^\tau_n(\rho)=V_{N_0}(\rho)$ for all $n \ge N_0$ and $\tau \in \cT$. 
It thus suffices to show that $\Pr(K \subset B_{2N_0-2}) \to 1$ as $N_0 \to \infty$. For this, it suffices to show that $\Pr(|K|>k) \to 0$ as $k \to \infty$ uniformly in $N_0$.
To this end, we fix $N>N_0$, let $\cS=\cS_{N_0,N}$ be the union of $S^+_n \cup S^-_n$ over all $N_0 \le n \le N$, and let $\cK=\cK_{N_0,N}$ be the connected component of the root in $\cS \setminus T$. It now suffices to show that $\Pr(|\cK|>k) \to 0$ as $k \to \infty$ uniformly in $N_0$ and $N$.

The proof now proceeds along similar lines as that of \cref{lem:K-prob-bound}, with the main modification arising from the need to take into account that \cref{cl:boundary-values} does not hold for vertices in $L_{2n}$. We borrow some notations from the proof of \Cref{lem:K-prob-bound}, with suitable modifications. Recall that in that proof we exploited the fact that the cookie values were forced in large regions (outer boundaries of $A^\pm$ of suitable parity). In the current situation, some parts of the outer boundary belong to $T$, and thus will not have forced cookie values. Since $T$ is only a small $\eps$ proportion of the sites, this will turn out to be a minor issue.

Denote $\cS=\cS^+ \cup \cS^-$, where $\cS^\pm := \bigcup_{N_0 \le n \le N} S^\pm_n$. As in the previous proof, we fix $A=A^+ \cup A^- \subset B_{2N}$, and consider the events $\{ \cK \cap \cS^+ = A^+\}$ and $\{ \cK \cap \cS^- = A^-\}$ separately.
Consider first the event $\{\cK \cap \cS^+ = A^+\}$ and suppose it holds. Observe that the cookies are all $1$ on $\text{Even} \cap (\partial^* A^+) \setminus T$. Indeed, if an even vertex $u$ is not in $A^+ \cup T$ and it has a parent $p(u)$ in $A^+ \cap S^+_n$ for some $N_0 \le n \le N$, then it must be that $V_n(p(u))=1$ and $V_n(u)=0$, and the cookie on the edge connecting them must be 1 (since we minimize over the children of $u$ in the game). We shall gain in probability by showing that the set where the cookie values are forced is large. Indeed, note that every even vertex $v$ in $A^+$ has a child in $A^+$ (applying the first item of \Cref{cl:boundary-values}, recalling that $\cK$ does not intersect $T$, and $T$ is contained in the even vertices by definition)  so that we can apply \cref{lem:parity-iso**} to obtain that
\[ |\text{Even} \cap (\partial^* A^+) \setminus T| = |\text{Even} \cap \partial^* A^+| - |N(A^+) \cap T| \ge \tfrac12(d-1)|A^+| - |N(A^+) \cap T| .\]

Consider next the event $\{\cK \cap \cS^- = A^-\}$.
Similarly to before, we observe that on this event all cookies are $-1$ on $\text{Odd} \cap \partial^* A^-$. We want to show that this set is large. However, in order to apply our isoperimetric inequality, we would need that every odd vertex in $A^-$ has a child in $A^-$, which is not necessarily the case as we insisted not to include $T$ in $\cK$. To get around this, we define ${\sf D} := \{ v \in A^- : N(v) \cap T \neq \emptyset\}$, and observe that the set $A^- \setminus {\sf D}$ has the property that every odd vertex in it has a child in it. Hence, \cref{lem:parity-iso**} yields that
\[ |\text{Odd} \cap \partial^*(A^- \setminus {\sf D})| \ge \tfrac12(d-1)|A^- \setminus {\sf D}| ,\]
which gives that
\[ |\text{Odd} \cap \partial^* A^-| \ge |\text{Odd} \cap \partial^*(A^- \setminus {\sf D})| - |{\sf D}| \ge \tfrac12(d-1)|A^-| - \tfrac12(d+1)|{\sf D}| .\]

Thus, {on the event $\{ \cK \cap \cS^+ = A^+\} \cap \{ \cK \cap \cS^- = A^-\}$,} the cookie values are forced on a predetermined set of size at least $\frac12(d-1)|A| - |N(A^+) \cap T|- \frac12(d+1)|{\sf D}|$.
Summing over the choices of $N(A^+) \cap T$ and ${\sf D}$, we get that
\begin{align*}
 \Pr(\cK \cap \cS^+ = A^+,~\cK \cap \cS^- = A^-)
 &= \sum_{F,D} \Pr(\cK \cap \cS^+ = A^+,~ \cK \cap \cS^- = A^-,~N(A^+) \cap T = F,~{\sf D}=D) \\
 &\le 2^{-\frac12(d-1)|A|} \cdot \sum_{F,D} \Pr(N(A^+) \cap T = F,~{\sf D}=D) \cdot 2^{|F| + \frac12(d+1)|D|} \\
 &= 2^{-\frac12(d-1)|A|} \cdot \E\left[ 2^{|N(A^+) \cap T|+\frac12(d+1)|{\sf D}|} \right] \\
 &\le 2^{-\frac12(d-1)|A|} \cdot \prod_{v \in A} \E\left[ 2^{|N(v) \cap T|+\frac12(d+1)\1_{\{N(v) \cap T \neq \emptyset\}}} \right] .
 \end{align*}
Now it is straightforward to see that the random variable inside the expectation is $1$ if $N(v) \cap T$ is empty and is $O(2^{d/2})$ if $N(v)\cap T$ is a singleton. The lower order terms in $\eps$ do not contribute and we end up with $1+\eps O(d2^{d/2})$ which is enough for our purposes. Below is a detailed calculation.
 
The expectation in the product is zero unless $v$ is in $L_{2N_0-1} \cup L_{2N_0+1} \cup \cdots$, in which case, 
\begin{align*}
   \E\left[ 2^{|N(v) \cap T|+\frac12(d+1)\1_{\{N(v) \cap T \neq \emptyset\}}} \right]
   &= \E\left[ \1_{\{N(v) \cap T = \emptyset\}} \right] + 2^{\frac12(d+1)} \cdot \E\left[ 2^{|N(v) \cap T|} \cdot \1_{\{N(v) \cap T \neq \emptyset\}} \right] \\   
   &{= (1-\eps)^d + 2^{\frac12(d+1)}\sum_{k=1}^d \binom dk \eps^k (1-\eps)^{d-k} 2^k } \\
   &= (1-\eps)^d + 2^{\frac12(d+1)}((1+\eps)^d-(1-\eps)^d) \\
   &\le 1+4\eps d 2^{\frac12(d+1)},
\end{align*}
where in the last inequality we used that $(1+x)^d-(1-x)^d \le 4dx$ for $0 \le x \le \frac1d$.
Hence,
\[ \Pr(\cK \cap \cS^+ = A^+,~\cK \cap \cS^- = A^-) \le \big(2^{-\frac12(d-1)} + 8\eps d \big)^{|A|}, \]
Summing over the choices of $A^-$, $A^+$ and $A$ (as in the proof of \Cref{lem:K-prob-bound}) leads to the bound:
\[ \Pr(|\cK| = k) \le \left((3ed) \cdot \big(2^{-\frac12(d-1)} + 8\eps d\big)\right)^k .\]
Since this bound does not depend on $N_0$ and $N$, and the expression in the base of the exponent is less than 1 when $d \ge 15$ and $\eps$ is small enough, we conclude that $\Pr(|\cK|>k)$ tends to zero uniformly in $N_0$ and $N$. Hence, almost surely, $V^\tau_n(\rho)=V_n=V_\infty$ for all $n$ large enough and all $\tau \in \cT$.
\end{proof}

The proof of \cref{thm:stability2} yields an explicit estimate on $\eps$, which satisfies, in particular, that $\eps \ge c/d^2$ for some constant $c>0$ and all $d$ large enough. This yields the same estimate for the $\eps$ in \cref{thm:stability}. The following result bootstraps this to order $1/d$ up to a logarithmic factor, and further shows that $\eps$ can in fact tend to 1 if one restricts to ``one-sided'' boundary conditions.

\begin{thm}\label{thm:bootstrap}
   Let $d$ be large enough. Let $F_\infty$ be the CDF of $V_\infty$, and let $F$ be an even CDF having an atom of size $p$ at 0. Then $\Psi^n F \to F_\infty$ weakly as $n \to \infty$ whenever
   \[ p \ge 1 - \frac{\log d - \log(5\log d)}d .\]
   Moreover, if $F$ is supported on non-negative integers, then the same convergence holds whenever
   \[ p \ge \frac{7\log d}d .\]
\end{thm}
\begin{proof}
Let $F$ be a CDF having an atom of size $p$ at 0. Recall that $\Psi = \Phi \circ \Phi'$, where $\Phi$ and $\Phi'$ correspond to the max and min steps, respectively.
Let $Z_0,Z_1,Z_2$ be distributed according to $\Psi F, \Phi'F,F$, respectively.
Let $X_1$ and $X_2$ denote the cookies of some vertex at height 1 and 2, respectively.
We have
\[ \Pr(Z_1=-1) \ge \Pr(Z_2=0)^d - \Pr(Z_2=0,X_2=1)^d = p^d(1-2^{-d}) \]
and
\[ \Pr(Z_1>-1) = \Pr(Z_1 \ge 1) \le (1 - \Pr(Z_2=0,X_2=-1))^d = (1-\tfrac12p)^d .\]
Therefore, 
\begin{align*}
 \Pr(Z_0=0)
  &\ge \Pr(Z_1 \le -1)^d - (1-\Pr(Z_1=-1, X_1=1))^d \\
  &\ge \big(1-(1-\tfrac12p)^d\big)^d - \big(1-\tfrac12p^d(1-2^{-d})\big)^d .
\end{align*}
Thus, if $p \ge 1-(\log d - \log(5\log d))/d$, then $\Pr(Z_0=0) \ge 1-1/d^{2.5+o(1)}$, showing that $\Psi F$ satisfies the assumption of \cref{thm:stability} (using that $\eps \approx d^{-2}$). We conclude that, for large $d$, if $p \ge 1-(\log d - \log(5\log d))/d$, then $\Psi^n F \to F_\infty$.

Suppose now that $F$ is supported on non-negative integers.
In this case, $Z_1 \ge -1$ so that
\begin{align*}
 \Pr(Z_0=0)
  &\ge \Pr(Z_1 = -1)^d - \Pr(Z_1=-1, X_1=-1)^d \\
  &\ge \Pr(Z_1 \le -1)^d \cdot (1-2^{-d}) \\
  &\ge 1- d(1-\tfrac12p)^d - 2^{-d} .
\end{align*}
Now, if we just assume that $p \ge 7(\log d)/d$, then we get that $\Pr(Z_0=0) \ge 1-1/d^{2.5}-2^{-d}$. We conclude as before that, for large $d$, if $p \ge 7(\log d)/d$, then $\Psi^n F \to F_\infty$.
\end{proof}

\begin{remark}[Stability of the fixed point]\label{rem:stability}
The results of this section imply a strong form of stability of the fixed point $F_\infty$ when $d$ is large enough. Indeed, by \cref{lem:K-prob-bound} (and the simple fact that $V_\infty=0$ whenever $\rho \notin K$), we see that $\Pr(V_\infty=0) \ge 1 - e^{-\Omega(d)}$. Thus, if $F$ is close to $F_\infty$, say in sup norm (much less is actually needed), then \cref{thm:bootstrap} implies that $\Psi^n F \to F_\infty$. In this sense, $F_\infty$ is a stable fixed point for $\Psi$.
\end{remark}

\begin{remark}[Phase transition in boundary conditions]
  \label{rem:random-boundary}
  \cref{thm:stability} implies a certain phase transition for the min-max game on a tree of large degree $d$.
  Consider the $n$-round game in which the boundary values are given by independent Bernoulli random variables, which take values 0 or 2 with probability $p$ and $1-p$, respectively.
  Let $F_\infty$ denote the CDF of $V_\infty$, and $F'_\infty$ the CDF of $V_\infty+2$ (note that this corresponds to the limiting value for the game with all 2 boundary conditions). \cref{thm:stability} implies that, for $d \ge 15$, the value of the game converges in distribution to $F_\infty$ when $p$ is sufficiently close to 1, and to $F'_\infty$ when $p$ is sufficient close to 0.
  This suggests that there is a critical value $p_c$ which marks the transition between convergence to $F'_\infty$ and $F_\infty$.
  Numerical computations applying $\Psi$ repeatedly to different initial distributions also suggest that the limit flips from $F_\infty$ to $F'_\infty$ at some critical $p$.
  However, we do not know that there is a unique transition point, and have not ruled out that there are multiple or even a continuum of limits as $p$ varies in $[0,1]$.
  We do know from \cref{thm:bootstrap}, that for $d$ large enough any such  critical value would satisfy
  \[ \frac{\log d - \log(5\log d)}d \le p_c \le \frac{7\log d}d .\]

  With a bit more work, it is possible to show that $p_c = \frac{(2+o(1))\log d}{d}$.
  To give a simple heuristic that suggests this, note that for small $p$ and large $d$, the value $V_n(u)$ for a vertex at level $2n-1$ has probability $1-(1-p/2)^d$ to be $-1$, and otherwise is very likely to be $+1$, with a much smaller probability $(1-p)^d/2^d$ of being 3.
  If we ignore the 3, and negate the values (to account for the max player moving in the next level), we find that the fixed point of this operation to be when $p=(1-p/2)^d$.
  If $p=\frac{\alpha\log d}{d}$ this becomes $\alpha\log d = d^{1-\alpha/2}$, hence $\alpha\approx 2-\frac{2\log\log d}{\log d}$.
  The solution of $p=(1-p/2)^d$ agrees extremely well with $p_c$ for all $d$: It is off by less than 3\% relative error for $d=3$, and less than $10^{-4}$ for $d=10$. 

  If the leaf distributions interpolate between all $0$ and all $2k$ for some larger value of $k$, the limit $\Psi^n F$ seems to have $k$ phase transitions.
\end{remark}

\subsection{Isoperimetry}

For $A \subset \T_d$, denote by $\partial^* A := N(A) \setminus A$ the set of sites not in $A$ whose parent is in $A$. We prove the following isoperimetric inequality.

\begin{lemma}\label{lem:parity-iso**}
    Let $A \subset \T_d$ be a finite set such that $N(v) \cap A \neq \emptyset$ for all even $v \in A$. Then
    \[ |\partial^* A \cap \text{Even}| \ge (d-1)|A \cap \text{Odd}| \ge \tfrac12 (d-1)|A| .\]
    An analogous statement holds when the roles of even and odd are reversed.
\end{lemma}

Note that if $A$ is a connected set containing the root, then $\partial^* A$ is just the usual external vertex boundary of $A$, denoted $\partial A$. We first need the following intermediate lemma.

\begin{lemma}\label{lem:parity-iso}
    Let $A \subset \T_d$ be a finite connected set containing the root. Then
    \begin{align*}
     |\partial A \cap \text{Even}| &= d|A \cap \text{Odd}| - |A \cap \text{Even}| + 1,\\
     |\partial A \cap \text{Odd}| &= d|A \cap \text{Even}| - |A \cap \text{Odd}| .
    \end{align*}
\end{lemma}
\begin{proof}
One way to see this is by the following observation. Let $v \in A$ and let $N_A(v)$ be the set of children of $v$ not in $A$. Then $N_A(v)$ must be in $\partial A$ and all of $N(v)\setminus N_A(v)$ are in $A$ and are of opposite parity to that of $v$. For the first equation, we add 1 to the right-hand side since the root vertex is not a child of any vertex.

We prove this formally by induction on $|A|$.  For $|A|=1$, $A$ is the root vertex, and this is trivial. Suppose that $|A|>1$. Let $v \in A \setminus \{\rho\}$ be such that $N(v) \cap A = \emptyset$, and denote $A' := A \setminus \{v\}$. If $v$ is odd, we have
\begin{align*}
 |A \cap \text{Even}| &= |A' \cap \text{Even}|,\\
 |A \cap \text{Odd}| &= |A' \cap \text{Odd}|+1,\\
 |\partial A \cap \text{Even}| &= |\partial A' \cap \text{Even}|+d,\\
 |\partial A \cap \text{Odd}| &= |\partial A' \cap \text{Odd}|-1 .
\end{align*}
From this and the induction hypothesis applied to $A'$, we obtain the desired equality. A similar argument applies when $v$ is even.
\end{proof}

\begin{lemma}\label{lem:parity-iso*}
    Let $A \subset \T_d$ be a finite set with $k_{\text{e}}$ (resp.\ $k_{\text{o}}$) connected components whose topmost element is even (resp.\ odd). Then
    \begin{align*}
     |\partial^* A \cap \text{Even}| &= d|A \cap \text{Odd}| - |A \cap \text{Even}| + k_{\text{e}},\\
     |\partial^* A \cap \text{Odd}| &= d|A \cap \text{Even}| - |A \cap \text{Odd}| + k_{\text{o}} .
    \end{align*}
\end{lemma}
\begin{proof}
    Let $A_1,\dots,A_{k_{\text{e}}+k_{\text{o}}}$ be the connected components of $A$. Apply \cref{lem:parity-iso} to each $A_i$ in the $d$-ary subtree rooted at the topmost element of $A_i$, noting that the roles of even and odd are reversed when the topmost element is odd. Upon summation the lemma is obtained.
\end{proof}

\begin{proof}[Proof of \cref{lem:parity-iso**}]
    The assumption of the lemma implies that $|A \cap \text{Odd}| \ge |A \cap \text{Even}|$. \cref{lem:parity-iso*} now gives the lemma.
\end{proof}

\subsection{Proof of \cref{thm:asymptotics}}
Suppose $d$ is large. Recall that the limit $V_\infty(v):=\lim_{n \to \infty} V_n(v)$ exists almost surely for all $v$, and that $V_\infty = V_\infty(\rho)$. Denote
\[ p := \Pr(V_\infty \neq 0) .\]
Clearly, $\Pr(V_\infty(v) \neq 0)=p$ for every even vertex $v$.
Moreover, by \cref{thm:R-symmetry}, $\Pr(V_\infty(v) \neq -1)=p$ for every odd vertex $v$.

Recall from~\eqref{eq:S+-_n} that $S^\pm_n$ are the sets of vertices with ``atypically'' high/low values in the $n$-round game. Let $K_n$ be the connected component of $S_n := S^+_n \cup S^-_n$ containing the root. Observe that
\[ p = \lim_{n \to \infty}\Pr(\rho \in S_n) = \lim_{n \to \infty}\Pr(|K_n| \ge 1) .\]
Recall that $K$ is the connected component of $S := \bigcup_n S_n$ containing the root.
Since $K_n \subset K$, \cref{lem:K-prob-bound} gives that
\[ p \le \Pr(|K| \ge 1) \le O(d 2^{-d/2}) .\]
We may improve on this by fixing $n$ to be large and considering small sizes of $K_n$ more carefully and appealing to the previous lemma for larger sizes (all the bounds below hold uniformly in $n$). By \cref{lem:K-prob-bound}, we have that
\[ \Pr(|K_n| \ge 5) \le \Pr(|K| \ge 5) \le O(d^5 2^{-2.5d}) .\]
For $|K_n|=3,4$, it is not hard to see (e.g., by \cref{cl:boundary-values}) that 
\[ \Pr(|K_n|=3) = \Theta(d^2 2^{-2d}) \qquad\text{and}\qquad \Pr(|K_n|=4) = \Theta(d^3 2^{-2d}) .\]
Let us elaborate on the $|K_n|=3$ estimate, the $|K_n|=4$ case being similar. If $|K_n|=3$, then it is not hard to see that there are $O(d^2)$ choices of the set $K_n$ and each such choices forces $2d-O(1)$ cookies. To see this, let $A=N(\rho) \cap K_n$ and $B=N(N(\rho))\cap K_n$. There are two cases: $|A|=2$ or $|A|=1, |B|=1$. In the former case, $K_n \subset S_n^+$, $V_n(\rho)=2$, and $V_n(u)=1$ for $u \in A$. Also, $V_n(w)=0$ for $w \in N(A)$, so that all the cookies in $X_{N(A)}$ must be $1$, which has probability $2^{-2d}$. Combine this with the fact that there are $\Theta(d^2)$ choices for $A$ to complete the estimate for this case. For the $|A|=1,|B|=1$ case, for similar reasons, $K_n \subset S_n^-$, $V_n(\rho)=-2$, $V_n(u)=-3$ for $u \in A$, and $V_n(w)=-2$ for $w \in B$, and the cookies in $N(\rho)\setminus A$ and those in $N(B)$ are all $-1$. The rest follows similarly.

Thus,
\begin{equation}\label{eq:p-bound}
 p = \Pr(1 \le |K_n| \le 2) + \Theta(d^3 2^{-2d}) .
\end{equation}
Observe that
\[ \{|K_n|=1\} = \{ X_{N(\rho)} \equiv -1\text{ and } S_n \cap N(\rho) = \emptyset \} .\]
Since $N(\rho)$ is a set of $d$ odd sites whose subtrees are pairwise disjoint, we have that
\[ \Pr(|K_n|=1) = 2^{-d} (1-p)^d .\]
Observe next that
\begin{equation}\label{eq:K=2}
 \{|K_n|=2\} = \biguplus_{u \in N(\rho)} \big\{ X_{N(u) \cup \{u\}} \equiv 1\text{ and } S_n \cap N(\{\rho,u\}) \setminus \{u\} = \emptyset \big\} .
\end{equation}
Thus,
\[ \Pr(|K_n|=2) = d 2^{-d-1} (1-p)^{2d-1} .\]
In particular, the probability that $1 \le |K_n| \le 2$ is at most $O(d2^{-d})$. Together with~\eqref{eq:p-bound}, this already yields the improved bound $p \le O(d2^{-d})$.
Using this, we obtain that
\begin{align*}
 \Pr(|K_n|=1) &= 2^{-d} - O(d^2 2^{-2d}),\\
 \Pr(|K_n|=2) &= d2^{-d-1} - O(d^3 2^{-2d}) .
\end{align*}
Plugging this back into~\eqref{eq:p-bound}, we obtain that
\[ \Pr(K_n=\emptyset) = 1 - p = 1 - d2^{-d-1} - 2^{-d} \pm O(d^3 2^{-2d}) . \]

We are now ready to estimate the distribution of $V_\infty$.
To this end, we estimate the distribution of $V_n$ for $n$ large.
First note that $\{V_n=0\}=\{K_n=\emptyset\}$, which gives the estimate for $\Pr(V_\infty=0)$. Next, observe that $\{|K_n|=1\} \subset \{V_n=-2\}$ and $\{|K_n|=2\} \subset \{V_n=2\}$. Thus,
\[ \Pr(V_n=-2) = \Pr(|K_n|=1)+\Pr(V_n=-2,|K_n|\ge3) = 2^{-d} \pm O(d^3 2^{-2d}) .\]
and
\[ \Pr(V_n=2) = \Pr(|K_n|=2)+\Pr(V_n=2,|K_n|\ge3) = d2^{-d-1} \pm O(d^3 2^{-2d}) .\]
This implies the claimed estimates for $\Pr(V_\infty=-2)$ and $\Pr(V_\infty=2)$.
Finally, note that $|V_n| \ge 4$ implies that $|K_n| \ge 2d$. Let us argue this for $V_n \ge 4$, the case $V_n \le -4$ being similar. If $V_n \ge 4$, then there is at least one height 1 vertex whose value is at least 3. Each such vertex has all their children having value at least $2$, so there are at least $d$ height 2 vertices having value at least 2. Finally, each such vertex has at least one child with value at least 1. Thus, there are at least $d$ height 1 vertices in $K_n$ and at least $d$ height 3 vertices in $K_n$. In particular, $|K_n| \ge 2d$.
Now applying \Cref{lem:K-prob-bound},
\[ \Pr(|V_n| \ge 4) \le (Cd)^d 2^{-d^2} \le e^{-\Omega(d^2)} .\]
This proves the claimed estimate on $\Pr(|V_\infty| \ge 4)$.

It remains to compute the expectation of $V_\infty$.
By \cref{lem:K-prob-bound},
\[ \E[|V_\infty| \1_{\{|V_\infty| \ge 4\}}] \le e^{-\Omega(d^2)} .\]
By our previous estimates,
\begin{align*}
 \E[V_\infty \1_{\{|V_\infty| \le 2\}}]
 &=2 \cdot [d2^{-d-1} \pm O(d^3 2^{-2d})] \\
 &-2 \cdot [2^{-d} \pm O(d^3 2^{-2d})] \\
 &= (d-2) 2^{-d} \pm O(d^3 2^{-2d}).
\end{align*}
Thus,
\[ \E V_\infty = (d-2) 2^{-d} \pm O(d^3 2^{-2d}) .\]

\section{Convergence in law for low $d$}
\label{sec:convergence-in-law}

In this section we prove \Cref{thm:distributional-convergence} for $3 \le d \le 15$.
Recall that $F_n$ is even, and so we can (and will) regard it as an element of $[0,1]^{2\Z}$.
The main strategy is to think of $(F_{2n})_{n \ge 1}$ as a sequence in the Banach space $\ell_\infty(2\Z)$, and show that $\Psi$ is a contraction. To that end, we need to estimate the norm of the derivative of $\Psi$. However, as it turns out, the derivative of $\Psi$ does not have norm strictly less than $1$ uniformly over all of $\ell_\infty(2\Z)$. Indeed, this is necessarily the case as there is more than one fixed point by shift invariance (if $F'$ is a fixed point, then so is $F''$ given by $F''(i)=F'(i+2)$). 
Our modified strategy is to show that $\Psi$ restricted to some small convex set has derivative less than 1 and is thus a contraction.
Moreover, we show that after some number of iterations, the CDF enters and is confined to this set.

The above strategy works as outlined for $d\ge 5$, but as it turns out, for $d=3,4$ there is an extra layer of complicacy. In these cases, inside the relevant small convex set, the norm of the derivative is bigger than 1, but the largest eigenvalue of the derivative operator is strictly less than 1. So we need to appropriately conjugate the derivative of $\Psi$ by a linear map (equivalently apply a invertible, bounded linear transformation to  $\Psi$), so that the derivative of the transformed operator has norm strictly less than 1. In practice, we apply the conjugation for all $3 \le d \le 15$ as it leads to better bounds.
In all these cases the proofs are computer assisted. We use interval arithmetic to formally rule out numerical errors.

We organize the proof as follows.
We first collect some sufficient conditions so that our contraction argument goes through, reducing the problem to a finite-dimensional problem. Then we verify the sufficient conditions using interval arithmetic in a separate \Cref{sec:conjugation_existence}.

\subsection{Reducing to a finite-dimensional problem}\label{sec:finite_dimensional}
The key proposition is the following. Let $D\Psi|_x$ denote the Fr\'echet derivative of $\Psi$ evaluated at $x$.
\begin{prop}\label{prop:derivative}
Fix $3 \le d  \le 15$.
There exists a constant $c_* \in (0,1)$, an invertible bounded linear map $E:\ell_\infty(2\Z) \to \ell_\infty(2\Z)$, a convex set $U \subset [0,1]^{2\Z}$, and some $N \ge 1$ such that 
\begin{enumerate}[label=\alph*.]
    \item $F_{2n} \in U$ for all $n \ge N$,
    \item $\|E (D\Psi)|_x E^{-1}\| \le c_*$ for all $x \in U$.
\end{enumerate}
\end{prop}

We now finish the proof of \Cref{thm:distributional-convergence}
using \Cref{prop:derivative}.

\begin{proof}[Proof of \Cref{thm:distributional-convergence} (for $3 \le d \le 15$)
assuming \Cref{prop:derivative}]
  Define $G_n := E F_n$ for $n \ge 0$. Then, by the mean value theorem, for any $n>N$, there exists $x \in EU$ such that
\begin{align*}
 \|G_{2n+2} - G_{2n}\|
  &= \|E\Psi E^{-1} G_{2n} - E\Psi E^{-1} G_{2n-2}\| \\
  &= \| (D(E\Psi E^{-1}))|_x (G_{2n}-G_{2n-2}) \| \\
  &= \| E (D\Psi)|_{E^{-1} x} E^{-1} (G_{2n}-G_{2n-2}) \| \\
  &\le c_* \|G_{2n}-G_{2n-2}\| .
\end{align*}   
Therefore, $G_{2n}$ is a Cauchy sequence in a Banach space, hence it converges.
Since $E^{-1}$ is also bounded, $F_{2n} = E^{-1} G_{2n}$ also converges.
\end{proof}

Let us move on to an explicit computation of the derivatives.
First, for the convenience of the reader, let us recall (see \cref{sec:def}) that $F_{2n+2}=\Psi(F_{2n})$, where
\[ (\Psi(F))(i) = g(F(i-2),F(i), F(i+2)) \]
and
\[ g(x,y,z) = H(x,y,z)^d:=\left[1-\frac12 \left(1-\frac{x+y}{2}\right)^d  - \frac12\left(1-\frac{y+z}{2}\right)^d\right]^d .\]
It is straightforward to observe that $D \Psi|_x$ can be seen as a $2\Z\times2\Z$ matrix whose $(i,j)$-th entry  $(D \Psi|_x)_{i,j}$ is 0 when $|i-j|>2$ and otherwise 
\begin{equation}\label{eq:DPsi}
\begin{aligned}
    (D \Psi|_x)_{i,i-2}&= \frac{\partial g}{\partial x}|_{x_{i-2},x_i,x_{i+2}},\\
    (D \Psi|_x)_{i,i}&= \frac{\partial g}{\partial y}|_{x_{i-2},x_i,x_{i+2}},\\
    (D \Psi|_x)_{i,i+2}&= \frac{\partial g}{\partial z}|_{x_{i-2},x_i,x_{i+2}}.
\end{aligned}
\end{equation}
The derivative of $g$ is given by 

\begin{align}\label{eq:dg}
\dfrac{\partial g}{\partial x} 
& =  \dfrac{d^{2}}{4}\,
H(x,y,z)^{\,d-1}
j(x,y), \\
\dfrac{\partial g}{\partial y} 
& =  \dfrac{d^{2}}{4}
H(x,y,z)^{\,d-1}
(j(x,y)+j(y,z)), \\
\dfrac{\partial g}{\partial z} 
& =  \dfrac{d^{2}}{4}
H(x,y,z)^{\,d-1}
j(y,z),
\end{align}

where
\[
 j(x,y) := \left(1 - \dfrac{x+y}{2}\right)^{d-1}.
\]

Since $H \colon [0,1]^3 \to [0,1]$ and $j \colon [0,1]^2 \to [0,1]$, all the entries of $D\Psi|_x$ are non-negative when $x \in [0,1]^{2\Z}$.

\begin{lemma}\label{lem:upper_bound_derivative}
    If $x_i \in [a_i,b_i] \subset [0,1]$ for $i \in \{1,2,3\}$, then 
    \begin{align*}
    \frac{d^2}{4}H^{d-1}(a_1,a_2,a_3) j(b_1,b_2)& \le \frac{\partial g}{\partial x}\mid_{(x_1,x_2,x_3)} \le \frac{d^2}{4}H^{d-1}(b_1,b_2,b_3) j(a_1,a_2),\\
        \frac{d^2}{4}H^{d-1}(a_1,a_2,a_3)\left[j(b_1,b_2)+j(b_2,b_3)\right]& \le \frac{\partial g}{\partial y}\mid_{(x_1,x_2,x_3)} \le \frac{d^2}{4}H^{d-1}(b_1,b_2,b_3)\left[j(a_1,a_2)+j(a_2,a_3)\right],\\
\frac{d^2}{4}H^{d-1}(a_1,a_2,a_3) j(b_2,b_3)& \le \frac{\partial g}{\partial z}\mid_{(x_1,x_2,x_3)} \le \frac{d^2}{4}H^{d-1}(b_1,b_2,b_3) j(a_2,a_3).
    \end{align*}
\end{lemma}
\begin{proof}
    This is clear since $H$ is non-decreasing in each of its arguments and $j$ is non-increasing in each of its arguments.
\end{proof}

\begin{lemma}\label{lem:row_sum_outside}
    Let $x \in [0,1]^{2\Z}$ be the values of a CDF, and let $i_0,i_1 \in 2\Z$.
    \begin{enumerate}
        \item If $x(i_0) \ge a$, then for all $i>i_0$, the $i$-th row sum in $D\Psi|_x$ is upper bounded by 
    $$
    d^2(1-a)^{d-1}.
    $$ 
    \item If $x(i_1) \le b$, then for all $i<i_1$, the $i$-th row sum in $D\Psi|_x$ is upper bounded by
    $$
    d^2(1-(1-b)^d)^{d-1}.
    $$
    \end{enumerate}
\end{lemma}
\begin{proof}
    Suppose that $x(i_0) \ge a$ and $i>i_0$, $i \in 2\Z$.
    Then $x(i+2) \ge x(i) \ge x(i-2) \ge x(i_0) \ge a$, and we can apply \Cref{lem:upper_bound_derivative} to obtain that $D\Psi|_x(i) \le d^2(1-a)^{d-1}$, where we have used that $H(1,1,1)=1$ and $j(a,a)=(1-a)^{d-1}$.

    Suppose now that $x(i_1) \le b$ and $i<i_1$.
    Then $x(i-2) \le x(i) \le x(i+2) \le x(i_1) \le b$, and we can apply \Cref{lem:upper_bound_derivative} to obtain that $D\Psi|_x(i) \le d^2(1-(1-b)^d)^{d-1}$, where we have used that $j(0,0)=1$ and $H(b,b,b)=1-(1-b)^d$.
\end{proof}

We will approximate $D\Psi|_x$ by a finite-dimensional submatrix $D\Psi|_x^{\text{approx}}$.
Specifically, it will suffice for our purposes to consider the matrix $D\Psi|_x^{\text{approx}}$ spanned by the columns and rows indexed by $\{-4,-2,0,2,4\}$.
It will turn out that for large $n$, the values of $a$ and $b$ corresponding to $x(4)$ and $x(-4)$ that we will obtain are sufficiently close to 1 and 0 so that the bounds in \Cref{lem:row_sum_outside} are small. This will imply that the maximum row sum of $D\Psi|_x$ is attained as the maximum absolute row sum of $D\Psi|_x^{\text{approx}}$.
This will allow us to make use of the following lemma.

\begin{lemma}\label{lem:conjugation_condition}
    Let $\tilde E$ be a $5 \times 5$ invertible matrix with columns and rows indexed by $\{-4,2,0,2,4\}$. Let $E$ be the extension of $\tilde E$ to a bounded, invertible linear operator on $\ell_\infty(2\Z)$ defined by $E_{ij}=\tilde E_{ij} $ if $|i|,|j| \le 4$ and $E_{ij} = 1_{i=j}$ otherwise. Then for all image of CDFs $x \in [0,1]^{2\Z}$ with $x(4) \ge a$ and $x(-4) \le b$, we have that
    $$
    \|E^{-1} (D\Psi)|_x E \|_\infty \le \|\tilde E^{-1} (D\Psi)|_x^{\text{approx}} \tilde E\|_\infty +5E_{\max} \cdot \eps_{a,b},
   $$
   where $E_{\max}$ is the maximum absolute value of all entries in $\tilde E$  and $\tilde E^{-1}$, and
   \[ \eps_{a,b} := d^2(1-a)^{d-1} + d^2(1-(1-b)^d)^{d-1} .\]
\end{lemma}
\begin{proof}
    Write $D\Psi|_x$ as
\[
D\Psi_x = \begin{pmatrix}
D_{11} & D_{12} & 0 \\
D_{21} & D_{22} & D_{23} \\
0 & D_{32} & D_{33}
\end{pmatrix}
\]
where $D_{22} = D\Psi|_x^{\text{approx}}$ and the other entries are of appropriate dimensions, i.e., $D_{12}$ and $D_{32}$ have the same number of columns as $D_{22}$, and $D_{21}$ and $D_{23}$ have the same number of rows as $D_{22}$ (note that $D_{22}$ is a finite matrix, while all others are infinite). Then we can write the matrix form of the operator $E^{-1} (D\Psi|_x) E$ as 
\[
E^{-1} (D\Psi|_x) E = \begin{pmatrix}
D_{11} & D_{12}\tilde E & 0 \\
\tilde E^{-1}D_{21} & \tilde E^{-1}D_{22}\tilde E & \tilde E^{-1}D_{23} \\
0 & D_{32}\tilde E & D_{33}
\end{pmatrix}
\]
We have indexed the matrix by $2\Z$ with the third row and third column of $D_{22}$ indexed by 0. Also recall that the $l_\infty$ norm of a matrix is the maximum of the row sum of the absolute value of the entries.

 Note that except for $D_{22}$, the absolute values of all the entries of $D_{ij}$ are bounded above by $\eps_{a,b}$ by \cref{lem:row_sum_outside}. Furthermore, the matrices are quite sparse, from which simple inspection leads to the error bound of $5E_{\max}\eps_{a,b}$ for the max absolute value row sum of the all the submatrices above except $\tilde E^{-1}D_{22}\tilde E$. For example, the absolute value row sums of $D_{11}$ and $D_{33}$ are bounded above by $\text{Err}$ as there are at most $3$ non-zero entries in each of their rows (very rough bound). The matrix $D_{21} $ has only one non-zero entry which is the top element of the last column. Therefore $\tilde E^{-1}D_{21}$ has max absolute value row sum given by $E_{\max}\eps_{a,b}$. Similar logic holds for $\tilde E^{-1}D_{23}$. For $D_{32}$, the only non-zero entry is in the final column of the top row, and hence $D_{32}\tilde E$ has max absolute value row sum upper bounded by $5E_{\max} \eps_{a,b}$ since $\tilde E$ is a $5 \times 5$ matrix. Finally, the max absolute value row sum of $\tilde E^{-1}D_{22}\tilde E$ is $\|\tilde E^{-1} (D\Psi)|_x^{\text{approx}} \tilde E\|_\infty$, simply by definition. This completes the proof of the lemma.
\end{proof}

\begin{claim}\label{cl:conjugation_existence}
    Let $3 \le d \le 15$. There exist five intervals $I_{-4},I_{-2},I_0,I_2,I_4 \subset [0,1]$, and a $5\times 5$ invertible matrix $\tilde E$ indexed by $\{-4,2,0,2,4\}$ such that, with $a:=\min I_4$ and $b := \max I_{-4}$, the convex set
    \[ U := \big\{ x \in [0,1]^{2\Z} : x\text{ is a CDF with }x(i) \in I_i\text{ for }i \in \{-4,-2,0,2,4\} \big\} \] 
    satisfies that $F_{2n} \in U$ for all $n$ large enough and
    \[ c_* := \sup_{x \in U} \|\tilde E^{-1} (D\Psi)|_x^{\text{approx}} \tilde E\|_\infty +5E_{\max} \cdot \eps_{a,b} < 1 .\]
\end{claim}

We prove this claim in \cref{sec:conjugation_existence} below.
Let us finish this section by completing the proof of \cref{prop:derivative} assuming the above claim.

\begin{proof}
    [Proof of \Cref{prop:derivative} given \Cref{cl:conjugation_existence}]
    We choose $U$ and $c_*$ as in \Cref{cl:conjugation_existence}. Condition a. follows from the definition of $U$ and Condition b. follows from \Cref{lem:conjugation_condition}.
\end{proof}

\subsection{Proof of \cref{cl:conjugation_existence}} \label{sec:conjugation_existence}

The key concept used for the proof is known as interval arithmetic.
The core idea is as follows: if we iterate any function in a computer finite number of times, there are floating point errors, which  renders the output merely a non-rigorous estimate of the true value. To avoid this, interval arithmetic replaces the numbers with an interval (which contains the initial inputs). Then, instead of doing regular arithmetic using numbers, a version of arithmetic for intervals is used. This is done in a way such that the interval arithmetic output of the same iterations is an interval which is rigorously \emph{guaranteed} to contain the true value of the iteration. Of course, if the intervals end up being large after iterations, this approach is less useful.
We refer to  \Cref{sec:interval_arithmetic} for background and references on this topic.

Let $[\R]$ be the space of all closed bounded intervals.
We use the somewhat unconventional notation $I=[I.a,I.b]$ for $I \in [\R]$, so that $I.a$ and $I.b$ are the left and right endpoints of the interval $I$, respectively.
The reasoning behind this notation is that this matches the syntax in Python.  We will represent the truncated interval arithmetic version of $F_{2n}$ by a sequence $(z_{2n})_{n \ge 0}$ in $[\R]^{\{-4,-2,0,2,4\}}$. To define this sequence, set
\[ 
z_0:=([0,0], [0,0], [1,1], [1,1], [1,1]) ,
\]
and then inductively define $z_{2n+2}$ for $n \ge 0$ by
\begin{equation}
    z_{2n+2}(j) :=
    \begin{cases}
    g^{[\,]}_{\mathbb M}([0,z_{2n}(j).b],z_{2n}(j), z_{2n}(j+2) ) &\text{ if }j=-4\\
    g^{[\,]}_{\mathbb M}(z_{2n}(j-2),z_{2n}(j), z_{2n}(j+2)) &\text { if }j \in \{-2,0,2\}\\
    g^{[\,]}_{\mathbb M}(z_{2n}(j-2), z_{2n}(j), [z_{2n}(j).a,1]) &\text{ if }j=4
    \end{cases},
\end{equation}
where $g^{[\,]}_{\mathbb M}$ is the machine implementation of the interval arithmetic version of $g$ defined in \Cref{sec:interval_arithmetic}. For now, it suffices to mention that for $g(a,b,c) \in g^{[\,]}_{\mathbb M}(I_1,I_2,I_3)$ for any set of intervals $I_1,I_2,I_3 \in [\R]$ such that $a\in I_1,b \in I_2,c \in I_3$.

By the definition of $F_{2n}$ and monotonicity of $F_{2n}$, and the way interval arithmetic works (see \Cref{sec:interval_arithmetic} for details), 
\begin{equation*}
   F_{2n}(i) \in
   \begin{cases}
       z_{2n}(i) &\text{if }i \in \{-4,-2,0,2,4\}\\
       [0,z_{2n}(-4).b] &\text{if }i  <-4\\
       [z_{2n}(4).a,1] &\text{if }i>4.
    \end{cases}    
\end{equation*}
In this way, $z_{2n}$ accounts for the rounding and truncation errors when computing $F_{2n}$ in a computer, and thereby gives intervals which are guaranteed to contain the true value of $F_{2n}$. This leads to the natural extension of $z_{2n}$ to an element in $ [\R]^{2\Z}$ defined as follows:
\begin{equation*}
    Z_{2n}(i) := 
    \begin{cases}
        z_{2n}(i) &\text{if } i \in \{-4,-2,0,2,4\}\\
        [0,z_{2n}(-4).b] &\text{if }i<-4\\
        [z_{2n}(4).a,1] &\text{if }i>4
    \end{cases}.
\end{equation*}
With this definition, $Z_{2n}(i) \ni F_{2n}(i)$ for all $i\in \Z$. In this case, we simply write $F_{2n} \in Z_{2n}$.

Define $${\mathsf D}_{2n}:=D\Psi|_{F_{2n}}.$$
Since we know that $F_{2n}(i) \in Z_{2n}(i)$ for all $i$, we can now define an interval arithmetic approximation of ${\mathsf D}_{2n}$ by using~\eqref{eq:DPsi} and~\eqref{eq:dg} and employing the bounds in \Cref{lem:upper_bound_derivative}. This yields a matrix with entries in $[[0,1]]$ rather than $\R$. We denote this matrix by ${\mathsf D}^{[\,]}_{2n}$. To be more precise, ${\mathsf D}^{[\,]}_{2n}$ is the $2\Z\times 2\Z$ matrix with entries in $[[0,1]]$, whose $(i,j)$-th entry $({\mathsf D}^{[\,]}_{2n})_{i,j}$ is $[0,0]$ if $|i-j|>2$, and is the interval implicit in the bounds of \Cref{lem:upper_bound_derivative} (the first inequality when $j=i-2$, the second when $j=i$, and the third when $j=i+2$) under the substitution $[a_1,b_1]=Z_{2n}(i-2)$, $[a_2,b_2]=Z_{2n}(i)$, $[a_3,b_3]=Z_{2n}(i+2)$. It then holds that
\[ ({\mathsf D}_{2n})_{i,j} \in ({\mathsf D}^{[\,]}_{2n})_{i,j} \qquad\text{ for all }i,j \in 2\Z .\]

Let ${\mathsf D}_{2n,\text{approx}}$ be the submatrix of ${\mathsf D}_{2n}$ spanned by $\{-4,-2,0,2,4\}$ (this is $D\Psi|_{F_{2n}}^{\text{approx}}$ in the notation of the previous section), and let ${\mathsf D}^{[\,]}_{2n,\text{approx}}$ be the corresponding submatrix of ${\mathsf D}^{[\,]}_{2n}$.
Let 
\begin{equation}
  \text{Err}_{2n} := \eps_{z_{2n}(4).a,z_{2n}(-4).b}  \label{eq:error}
\end{equation}

For a finite matrix $A^{[\,]}$, with entries in $[\R]$, we denote 
\begin{equation}
    \|A^{[\,]}\| := \sup_i \sum_{j}\max\{|A_{ij}.a|, |A_{ij}.b|\},\label{eq:row_sum}
\end{equation}

The reason for this notation is that if $A^{[\,]}$ is an interval arithmetic approximation of some matrix $A$ (in the sense that $A_{i,j} \in A^{[\,]}_{i,j}$ for all $i,j$), then $\|A\|_{\infty}$ is clearly upper bounded by $\|A^{[\,]}\|$.
We also call $\|A^{[\,]}\|$ the \textbf{maximum absolute row sum} of $A^{[\,]}$.

\begin{fact}\label{fact:conjugation_existence}
  Let $3 \le d \le 15$.
  There exists $N$ such that the following holds:
  \begin{itemize}
  \item $Z_{2N+2} \subset Z_{2N}$ in the sense that $Z_{2N+2}(i) \subset Z_{2N}(i)$ for all $i \in 2\Z$.
  \item There is a real-valued $5 \times 5$ matrix $\tilde E$ such that
    \[ c_{**} := \|\tilde E^{-1}{\mathsf D}^{[\,]}_{2N,\text{approx}} \tilde E\| +5E_{\max} \cdot \text{Err}_{2N} < 1 .\]
  \end{itemize}
\end{fact}

Note $\tilde E^{-1}{\mathsf D}^{[\,]}_{2N,\text{approx}} \tilde E$ makes sense as a matrix with entries in $[\R]$, as matrix multiplication of $[\R]$-valued matrices is well defined by performing elementary operations (only $*$ and $+$ in this case, see \Cref{sec:interval_arithmetic}]) in interval arithmetic (and a real-valued matrix can be interpreted as a $[\R]$-valued matrix in the obvious manner).

\begin{proof}[Proof of \cref{cl:conjugation_existence} given \cref{fact:conjugation_existence}]
    Let $N$ and $\tilde E$ be as in the above fact. Take $I_i := Z_{2N}(i)$ for $i \in \{-4,-2,0,2,4\}$. Note that by the definition of $Z_{2n}$, this implies that the convex set $U$ defined in \cref{cl:conjugation_existence} is precisely $Z_{2N}$, interpreted as a subset of $[0,1]^{2\Z}$ by taking the Cartesian product of its intervals.
    By the construction of $Z_{2n}$, we have that $F_{2n} \in Z_{2n}$ for all $n$. Since $Z_{2N+2} \subset Z_{2N}$, we have that $Z_{2n} \subset Z_{2N}$ for all $n \ge N$, and, in particular, $F_{2n} \in Z_{2n} \subset Z_{2N} = U$.
    
    It remains to verify that $c_*$, defined in \cref{cl:conjugation_existence}, is less than 1. Indeed, for any $x \in U$, we have that ${\mathsf D}^{[\,]}_{2N,\text{approx}}$ is an interval arithmetic approximation of $(D\Psi)|_x^{\text{approx}}$, so that $\tilde E^{-1}{\mathsf D}^{[\,]}_{2N,\text{approx}} \tilde E$ is an interval arithmetic approximation of $\tilde E^{-1} (D\Psi)|_x^{\text{approx}} \tilde E$, and thus, $\|\tilde E^{-1} (D\Psi)|_x^{\text{approx}} \tilde E\| \le \|\tilde E^{-1}{\mathsf D}^{[\,]}_{2N,\text{approx}} \tilde E\|$. Hence, $c_* \le c_{**} < 1$.    
\end{proof}

We are left with verifying \cref{fact:conjugation_existence}, which we do below.

\subsubsection{Computer-assisted proof of \cref{fact:conjugation_existence}}
Up until this point, all our arguments didn't use a computer.
At this point, in order to prove \cref{fact:conjugation_existence} we make use of a computer.
As explained earlier, this too is entirely rigorous.

\begin{itemize}
    \item Using interval arithmetic, we compute the vector $z_{2N}$ for some large $N$ ($N=500$ suffices). This is the step where we iterate to enter a small convex set around the (potential) fixed point.
    \item Verify that
\begin{equation}
    \text{$z_{2N+2} \subset z_{2N}$}
    \label{eq:convex}
\end{equation}
by doing another round of iteration.
Thus, the first item of \Cref{fact:conjugation_existence} is verified.
    \item Calculate ${\mathsf D}^{[\,]}_{2N,\text{approx}}$ as described above. Recall that this is a $5 \times 5$ matrix with each entry being an interval, and is such that it is guaranteed to contain the central $5 \times 5$ submatrix of ${\mathsf D}_{2n}$ (which we called ${\mathsf D}_{2n,\text{approx}}$).
    \item We create a matrix with upper endpoints of the interval entries of ${\mathsf D}^{[\,]}_{2N,\text{approx}}$: let ${\mathsf B}^{[\,]}_{2N,\text{approx}}$ denote the real valued matrix whose $(i,j)$th entry is $d_{ij}.b$ where $d_{ij}$ is the $(i,j)$th entry of ${\mathsf D}^{[\,]}_{2N,\text{approx}}$. This matrix serves as a proxy for  ${\mathsf D}^{[\,]}_{2N,\text{approx}}$ since we could not find a package for computing eigenvalues in interval arithmetic.
    \item  Then we compute the matrix of eigenvectors of ${\mathsf B}^{[\,]}_{2N,\text{approx}}$, and call it $\tilde E$. Conjugating by $\tilde E$ brings the $l_\infty$ norm of ${\mathsf B}^{[\,]}_{2N,\text{approx}}$ closer to the largest eigenvalue of ${\mathsf D}_{2n, \text{approx}}$ (which will turn out to be strictly less than 1).

    \item We now want to upper bound the max absolute value row sum of $\|\tilde E^{-1}{\mathsf D}^{[\,]}_{2N,\text{approx}} \tilde E\|$. Computing inverse of a matrix  creates errors, we we get around that by using interval arithmetic again as follows.
    \item We convert the entries of $\tilde E$ into an interval valued matrix, call it $\tilde E^{[\,]}$. Then we compute the inverse of $\tilde E^{[\,]}$ using interval arithmetic. Call it $\tilde E^{-1,[\,]}$
    \item We point out that we do not need to be accurate with $\tilde E$ in the sense that we do not need it to be an exact eigenvector matrix  (and in fact the output in the computer won't be accurate). We just want some matrix which pushes the norm of the conjugated matrix to be strictly less than 1. The above choice of ${\mathsf B}^{[\,]}_{2N,\text{approx}}$ and $\tilde E$ does the job.
    
    \item The max absolute row sum of $\tilde E^{-1,[\,]}  {\mathsf D}^{[\,]}_{2N,\text{approx}} \tilde E^{[\,]}$ (call it $M$) upper bounds $\|\tilde E^{-1}{\mathsf D}^{[\,]}_{2N,\text{approx}} \tilde E\|$.
\item Now compute $\text{Err}_{2N}$ as defined in \eqref{eq:error}, $E_{\max}$ as defined in \Cref{lem:conjugation_condition} using $\tilde E$ and $\tilde E^{-1,[\,]}$ (we use an upper bound as in the summand of \eqref{eq:row_sum} for the interval valued matrix).
\item Verify that

\begin{equation}
 M+5\text{Err}_{2N}E_{\max} <1  \label{eq:verification}
\end{equation}
    \end{itemize}

We now present the outputs of the Python code which verifies \eqref{eq:verification} and \eqref{eq:convex} for $3 \le d \le 15$. The Python code are given in the Appendix at the end of the paper. These outputs completes the verifications \eqref{eq:verification} and \eqref{eq:convex} and hence completes the verification of \Cref{fact:conjugation_existence}.
\qed.
\begin{center}
\begin{tabular}{l c}
\hline
\textbf{Quantity} & \textbf{Value} \\
\hline
$d$ & $3$ \\
\hline
$z_{1000}[-4]$ & $[0.0011602871587319998,\; 0.0011993764698785406]$ \\
$z_{1000}[-2]$ & $[0.14874197863042835,\; 0.14907664678696686]$ \\
$z_{1000}[0]$ & $[0.7915785668954471,\; 0.7919913188108979]$ \\
$z_{1000}[2]$ & $[0.9982605777433734,\; 0.9982710593916913]$ \\
$z_{1000}[4]$  & $[0.999999999013227,\; 0.9999999990309593]$ \\
\hline
All intervals contained & True \\
\hline
$M$ & $0.9711041345$ \\
$\mathrm{Err}_{1000}$ & $0.00011623959660162087$ \\
$E_{\max}$ & $2.029643231$ \\
\hline
$\|\tilde E^{-1} {\mathsf D}_{1000} \tilde E\|_{\infty}$ 
& $\le 0.9722837591$ \\
\hline
\end{tabular}
\end{center}
\begin{center}
\begin{tabular}{l c}
\hline
\textbf{Quantity} & \textbf{Value} \\
\hline
$d$ & $4$ \\
\hline
$z_{1000}[-4]$ & $[3.677499508615473\times10^{-5},\; 3.684514518735852\times10^{-5}]$ \\
$z_{1000}[-2]$ & $[0.08280623285326534,\; 0.08280638746535361]$ \\
$z_{1000}[0]$ & $[0.8443270997063105,\; 0.8443272847307453]$ \\
$z_{1000}[2]$ & $[0.9999264520384025,\; 0.9999264523885499]$ \\
$z_{1000}[4]$  & $[0.9999999999999991,\; 1.0]$ \\
\hline
All intervals contained & True \\
\hline
$M$ & $0.8254689474$ \\
$\mathrm{Err}_{1000}$ & $5.121164963371005\times10^{-11}$ \\
$E_{\max}$ & $2.004955511$ \\
\hline
$\|\tilde E^{-1} {\mathsf D}_{1000} \tilde E\|_{\infty}$ 
& $\le 0.8254689479$ \\
\hline
\end{tabular}
\end{center}

To avoid repition, we fill in only the upper bound on $M+5\text{Err}_{1000}E_{\max}$ for the remaining rounded  up to 2 decimal places.
\begin{table}[h!]
\centering
\begin{tabular}{|c|c|c|c|c|c|c|c|c|c|c|c|c|}
\hline
$d$ & 3 & 4 & 5 & 6 & 7 & 8 & 9 & 10 & 11 & 12 & 13 & 14 \\
\hline
$\|{\mathsf D}_{1000}\|_{\infty} \le $ & .98 & .83 & 0.61 & 0.4    & .25 & .15 & .09 & .06 & .04 & .42 & .35 & .28 \\
\hline
\end{tabular}
\caption{Values of $\|D\Psi\|_{\infty}$ for $d = 3$ to $14$.}
\end{table}

\section{Discussion and open problems}

\paragraph{The binary tree.}
This is the least well understood case of this problem.
We have seen that the game played on a binary tree has properties distinct from that when played on a $d$-ary tree of degree $d \ge 3$.
In particular, in \cref{thm:fixed-points}, we have shown the existence of a continuous family of fixed points possessing a certain symmetry which is not possible when $d \ge 3$.
Simulations clearly suggest that $V_n$ converges in distribution to one of these fixed points.
It is not clear whether there is a.s. convergence for $d=2$.
See visualizations in
\href{https://www.math.ubc.ca/~angel/minmax_game}{\url{https://www.math.ubc.ca/~angel/minmax\_game}}.

\begin{problem}
  Show that $V_n$ converges in distribution when $d=2$.
\end{problem}

\begin{problem}
  Show that there are no solutions to the RDE beyond the ones we found.
\end{problem}

The known solutions to the RDE satisfy $F(x+2) = h(F(x))$ with $h(p)=2\sqrt{p}-p$.
While this is only relevant for even integers, one can iterpret this as follows.
Take any distribution satisfying the above recursion for all $x\in\R$.
If $Z$ has that such a law, rounding $Z$ down to an even integer will yield a discrete solution to the RDE.
However, it is easy to see that $Z$ itself is also a fixed point of the RDE.
There is much flexibility in the choice of such a distribution: any increasing function on $[0,2]$ with $F(2)=h(F(0))$ can be extended to all of $\R$.

\begin{problem}
  Is there an analytic CDF $F$ satisfying $F(x+2)=h(F(x))$ for all $x$?
  Even if yes, it is not unique, as it may be composed with any analytic homeomorphism $\xi$ of $\R$ with $\xi(x+2)=\xi(x)+2$.
  Possibly another assumption such as log-concavity may identify a unique canonical fixed point.
\end{problem}


\paragraph{Trees of degree three and higher.}
We have shown convergence in distribution of $V_n$ when $d \ge 3$, and almost-sure convergence when $d \ge 15$.

\begin{problem}
  Show that $V_n$ converges almost surely when $d \ge 3$.
\end{problem}

\begin{problem}
  Identify all fixed points of the RDE.
  We are only aware of the stable one and conjecturally the unstable one (see below), and their shifts.
  Are there any others?
\end{problem}

\paragraph{Random boundary conditions.}
We have focused on the game in which in the ``boundary conditions'' (the values at height $2n$ for the $n$-round game) are all 0.
Suppose we place boundary conditions which are 0 or 2, with probability $p$ and $1-p$, independently at each vertex.
We have seen in \cref{rem:random-boundary} that when $d$ is large, there is a transition from one limiting behavior (identical to that of the all 0 boundary condition) to another (the all 2 boundary condition), and that this occurs when $p$ is of the order $(\log d)/d$.
However, we have not established that there exists a critical point $p_c$ such that we get the former limit whenever $p<p_c$ and the latter limit whenever $p>p_c$.

\begin{problem}
  For $d$ large, show the existence of a critical point $p_c$.
\end{problem}

We believe that at the critical point there is a distinct limiting behavior.

\begin{problem}
  At $p=p_c$, show that the root value converges to an unstable fixed point distribution which is different from the $p=0$ or $p=1$ limit.
\end{problem}

Note that by definition, if the claimed fixed point at $p_c$ exists, it must be an unstable fixed point.
In contrast, for $d=2$, simulations suggest convergence to a continuous family of distributions as we vary $p$ (with the limit being one of the distributions described in the first item of \cref{thm:fixed-points}).
This is an indication of the nonexistence of a.s. convergence.

\paragraph{Other cookie distributions.}
We considered in this paper the game in which the edge cookies have distribution uniform on $\{-1,1\}$.
It is natural to consider this game with other distributions.
Devroye and Kamoun \cite{DK} have proved a law of large numbers type result $V_n\sim an$ for some $a$ depending on the law of the cookies.
Some of our methods extend to more general cookie distributions.
For example, our methods for large $d$ extend to cookie distributions which attain their essential infimum and supremum values with positive probability, but not otherwise.
In particular, we are unable to say anything about the convergence when the cookies are uniform on $[-1,1]$ (but also for discrete distributions which do not attain their extreme values, e.g. signed geometric).

\begin{problem}
  Does $V_n$ converge in distribution (or perhaps almost surely) when the cookie distribution is uniform on $[-1,1]$ and $d$ is large enough? What about Gaussian?
\end{problem}



 
\bibliographystyle{alpha}
\bibliography{minmax}

\begin{appendices}

\section{Background on interval arithmetic}
\label{sec:interval_arithmetic}

We provide a brief introduction to the theory of interval arithmetic. We refer to \cite[Chapter 1]{Neumaier_1991} for a more detailed exposition.
The primary goal of interval arithmetic is to eliminate errors while doing {computerized} arithmetic involving floating points. For example, suppose we want to evaluate  
$$
\frac1{12} = .083333\dots
$$
On the one hand, if we truncate the above decimal, it would introduce an error. On the other hand, the assertion
$$
\frac1{12} \in [.0832,.0834],
$$
is completely rigorous, but we pay a price by not specifying a more accurate representation of $1/12$. The idea now is to build an arithmetic for closed intervals rather than floating points, paying the price of sacrificing precision.

We define
$$[\R]= \{[a,b]: -\infty<a \le b<\infty\}.$$
to be the space of all closed bounded intervals. We identify $[a,a]$ with the real number $a$. We can now define the set of elementary operations $$\cO:=\{+,-,*,/,**\}$$
on $[\R]$ in a natural way (here $**$ stands for `power'). For every $\circ \in \cO$,
$$
[a,b]\circ[c,d] = \cH\{x \circ y: x \in [a,b], y \in [c,d]\}
$$
where $\cH$ denotes the convex hull, which in this setup is the smallest closed interval containing the set (division by an interval containing 0 is undefined). In fact, with this definition, we can write down formulas for addition and subtraction very easily:
$$
[a,b]+[c,d] = [a+c,b+d] \qquad [a,b]-[c,d] = [a-d,b-c].
$$
For multiplication and division, similar formulas can be written down, but we need to break into cases depending on the signs of the numbers involved. We refer to \cite[Table 1.1a,1.1b]{Neumaier_1991} for more details.

For our purposes, we need to employ interval arithmetic to a function $\phi$ from $\R^k$ to $\R$ which is a composition of elementary operations. Using the operations in the previous paragraph, we can naturally define the interval arithmetic version of $\phi$ by composing the interval arithmetic versions of the elementary operations. Denote by $\phi^{[\,]}$ the interval arithmetic version of $\phi$. Thus $\phi^{[\,]}: [\R]^k \to [\R]$ where $[\R]^k$ is a $k$-tuple of closed intervals, and it satisfies 
$$
\cH\{\phi(x_1,x_2,\ldots, x_k): x_i \in [a_i,b_i]\} \subseteq \phi^{[\,]}([a_1,b_1],[a_2,b_2],\dots, [a_k, b_k]).
$$

For employing interval arithmetic on a machine, more care is needed. This is because the set $\mathbb M$ of machine representable numbers is finite.
Most computers using binary encoding, so that even a number such as $1/12$ cannot be represented precisely.
Thus when we input an interval $[a,b]$ in the machine, it does an optimal outward rounding.
For example, $[a,b]$ is approximated by $[a_{\downarrow}, b_{\uparrow}]$ where $$a_\downarrow  = \sup \{x \le a: x \in \mathbb M\}, \text{ and }b_\uparrow = \inf \{x\ge b:x \in \mathbb M \}$$
So for any elementary operation $\circ \in \cO$ and $a,b,c,d \in \mathbb M$, the machine actually implements the operation $$[a,b] \bar \circ [c,d]  := [\alpha_{\downarrow}, \beta^{\uparrow}]$$  where $[\alpha, \beta] = [a,b] \circ [c,d]$.
Thus, in practice, when we ask the machine to implement $\phi$, it outputs $[\alpha_{\downarrow}, \beta^{\uparrow}]$ if $[\alpha, \beta]$ is the output of  $\phi^{[\,]}$. Let $\phi^{[\,]}_{\mathbb M}$ denote the function which yields the machine output of $\phi^{[\,]}$. The most important fact for us is that if $x_i \in [a_i,b_i] $ for $1\le i \le k$ then  $$\phi(x_1,\ldots,x_k) \in \phi^{[\,]}_{\mathbb M}([a_1,b_1],\dots, [a_k,b_k]) .$$
Sometimes, for reasons of speed, the Python package might not output the optimal interval. Nevertheless, what is important for us is that the output of the Python package is an interval which contains the optimal interval.
Consequently, the output $\phi(x_1,\ldots,x_k)$ is guaranteed to be in the output of the interval arithmetic version of $\phi$ in the Python output. We ignore this subtlety, and simply write $\phi_{\mathbb M}$ for the Python outputs as well.    

In practice, we use the iv.mpf package in Python3, we refer to \href{https://mpmath.org/doc/current/contexts.html}{this link} for the relevant documentation.

\section{Python Code}

We include here the code used to compute verify the bounds on the derivatives of $\Psi$ near the fixed point for small $d$.
This code is linked at \href{https://www.math.ubc.ca/~angel/minmax_game}{\url{https://www.math.ubc.ca/~angel/minmax\_game}}.

\begin{verbatim}
import math
import numpy as np
from scipy.linalg import eig
from mpmath import mp, iv

mp.dps = 10   # digits of precision after the decimal point
mp.pretty = True

d = 3 # change here to test d-ary tree

# ============================================================
# Initial conditions: CDF of delta_0
# ============================================================

a0 = iv.mpf([1, 1])
a2 = iv.mpf([1, 1])
a4 = iv.mpf([1, 1])

am2 = iv.mpf([0, 0])
am4 = iv.mpf([0, 0])

n = 5
num_iterations = 1000

interval = [am4, am2, a0, a2, a4]


# ============================================================
# Matrix storage
# ============================================================

rows, cols = n, n

# Derivative matrix D_{2n} in the paper
gradg = np.empty((rows, cols), dtype=object)

# Upper endpoints of gradg entries
gradg_upper = np.empty((rows, cols), dtype=float)

for i in range(rows):
    for j in range(cols):
        gradg[i, j] = iv.mpf([0, 0])

temp = interval[:]


# ============================================================
# Interval arithmetic helper functions
# ============================================================

def interval_abs_sup(I):
    """Return sup{|x| : x in interval I} for I = iv.mpf([a, b])."""
    a = mp.mpf(I.a)
    b = mp.mpf(I.b)
    return max(abs(a), abs(b))


def max_interval_abs_sup(A):
    """Return the largest sup-norm among entries of an interval matrix."""
    return max([interval_abs_sup(A[i,j])
                for i in range(A.rows)
                for j in range(A.cols)])
    
def float_to_iv_hull(x):
    """
    Take a float x and return a tiny interval guaranteed to contain it.
    This protects against floating-point storage error.
    """
    xf = float(x)
    lo = np.nextafter(xf, -np.inf)
    hi = np.nextafter(xf, np.inf)
    return iv.mpf([mp.mpf(lo), mp.mpf(hi)])


def to_iv_matrix(A):
    """
    Convert a numeric matrix into an interval matrix whose entries
    are guaranteed to contain the original floating-point values.
    """
    A = np.asarray(A, dtype=float)
    n_rows, n_cols = A.shape
    out = iv.matrix(n_rows, n_cols)

    for i in range(n_rows):
        for j in range(n_cols):
            out[i, j] = float_to_iv_hull(A[i, j])

    return out


# ============================================================
# Iteration map
# ============================================================

def g(x, y, z):
    return (
        1
        - 0.5 * (1 - (x + y) / 2) ** d
        - 0.5 * (1 - (y + z) / 2) ** d
    ) ** d


def iterate_intervals(intervals):
    """Apply one interval iteration step."""
    new_intervals = intervals[:]

    for i in range(len(intervals)):
        if i == 0:
            new_intervals[i] = g(iv.mpf([0, intervals[i].b]), intervals[i], intervals[i + 1])
        elif i == len(intervals) - 1:
            new_intervals[i] = g(intervals[i - 1], intervals[i], iv.mpf([intervals[i].a, 1]))
        else:
            new_intervals[i] = g(intervals[i - 1], intervals[i], intervals[i + 1])

    return new_intervals

# ============================================================
# Run interval iterations
# ============================================================

for _ in range(num_iterations):
    interval = iterate_intervals(interval)

print("Running test for d =",d)

#print(f"Basin after {num_interations} iterations:")
print(f"Interval basins for i in [-4,4] after {num_iterations} rounds:")
for i, val in interval:
    print(f"    [{float(val.a)}, {float(val.b)}]")


# ============================================================
# Check forward invariance: one more iteration
# ============================================================

temp = iterate_intervals(interval)

all_contained = True
for i in range(len(interval)):
    is_contained = (
        float(temp[i].a) >= float(interval[i].a)
        and float(temp[i].b) <= float(interval[i].b)
    )
    if not is_contained:
        all_contained = False

print(f"\nVerify image of basin is contained in the basin: {all_contained}\n")


# ============================================================
# Derivative analysis
# ============================================================

def H(x, y, z, d):
    return 1 - 0.5 * (1 - (x + y) / 2) ** d - 0.5 * (1 - (y + z) / 2) ** d

def dgdx(x1, x2, x3, y1, y2, y3, d):
    return iv.mpf([
        (d**2 / 4) * H(x1, x2, x3, d) ** (d - 1) * (1 - (y1 + y2) / 2) ** (d - 1),
        (d**2 / 4) * H(y1, y2, y3, d) ** (d - 1) * (1 - (x1 + x2) / 2) ** (d - 1),
    ])

def dgdy(x1, x2, x3, y1, y2, y3, d):
    return iv.mpf([
        (d**2 / 4) * H(x1, x2, x3, d) ** (d - 1) * (
            (1 - (y1 + y2) / 2) ** (d - 1)
            + (1 - (y2 + y3) / 2) ** (d - 1)
        ),
        (d**2 / 4) * H(y1, y2, y3, d) ** (d - 1) * (
            (1 - (x1 + x2) / 2) ** (d - 1)
            + (1 - (x2 + x3) / 2) ** (d - 1)
        ),
    ])

def dgdz(x1, x2, x3, y1, y2, y3, d):
    return iv.mpf([
        (d**2 / 4) * H(x1, x2, x3, d) ** (d - 1) * (1 - (y2 + y3) / 2) ** (d - 1),
        (d**2 / 4) * H(y1, y2, y3, d) ** (d - 1) * (1 - (x2 + x3) / 2) ** (d - 1),
    ])

# ============================================================
# Fill derivative matrix gradg
# Also record the errors from Lemma 5.3
# ============================================================

for i in range(len(interval)):
    if i == 0:
        gradg[i, i] = (
            dgdx(
                0,
                float(interval[i].a),
                float(interval[i + 1].a),
                float(interval[i].b),
                float(interval[i].b),
                float(interval[i + 1].b),
                d,
            )
            + dgdy(
                0,
                float(interval[i].a),
                float(interval[i + 1].a),
                float(interval[i].b),
                float(interval[i].b),
                float(interval[i + 1].b),
                d,
            )
        )

        gradg[i, i + 1] = dgdz(
            0,
            float(interval[i].a),
            float(interval[i + 1].a),
            float(interval[i].b),
            float(interval[i].b),
            float(interval[i + 1].b),
            d,
        )

        er_low = (
            dgdx(0, 0, 0, float(interval[i].b), float(interval[i].b), float(interval[i].b), d)
            + dgdy(0, 0, 0, float(interval[i].b), float(interval[i].b), float(interval[i].b), d)
            + dgdz(0, 0, 0, float(interval[i].b), float(interval[i].b), float(interval[i].b), d)
        )

    elif i == len(interval) - 1:
        gradg[i, i] = (
            dgdy(
                float(interval[i - 1].a),
                float(interval[i].a),
                float(interval[i].a),
                float(interval[i - 1].b),
                float(interval[i].b),
                1,
                d,
            )
            + dgdz(
                float(interval[i - 1].a),
                float(interval[i].a),
                float(interval[i].a),
                float(interval[i - 1].b),
                float(interval[i].b),
                1,
                d,
            )
        )

        gradg[i, i - 1] = dgdx(
            float(interval[i - 1].a),
            float(interval[i].a),
            float(interval[i].a),
            float(interval[i - 1].b),
            float(interval[i].b),
            1,
            d,
        )

        er_up = (
            dgdx(float(interval[i].a), float(interval[i].a), float(interval[i].a), 1, 1, 1, d)
            + dgdy(float(interval[i].a), float(interval[i].a), float(interval[i].a), 1, 1, 1, d)
            + dgdz(float(interval[i].a), float(interval[i].a), float(interval[i].a), 1, 1, 1, d)
        )

    else:
        gradg[i, i - 1] = dgdx(
            float(interval[i - 1].a),
            float(interval[i].a),
            float(interval[i + 1].a),
            float(interval[i - 1].b),
            float(interval[i].b),
            float(interval[i + 1].b),
            d,
        )

        gradg[i, i] = dgdy(
            float(interval[i - 1].a),
            float(interval[i].a),
            float(interval[i + 1].a),
            float(interval[i - 1].b),
            float(interval[i].b),
            float(interval[i + 1].b),
            d,
        )

        gradg[i, i + 1] = dgdz(
            float(interval[i - 1].a),
            float(interval[i].a),
            float(interval[i + 1].a),
            float(interval[i - 1].b),
            float(interval[i].b),
            float(interval[i + 1].b),
            d,
        )


# ============================================================
# Upper-endpoint proxy matrix for eigenvector computation
# ============================================================

for i in range(rows):
    for j in range(cols):
        gradg_upper[i, j] = float(gradg[i, j].b)

eigvals, V = eig(gradg_upper)

V_iv = to_iv_matrix(V)
V_inv = V_iv ** -1


# ============================================================
# Size controls for conjugation
# ============================================================

max_V = max_interval_abs_sup(V_iv)
max_V_inv = max_interval_abs_sup(V_inv)


# ============================================================
# Compute M = V^{-1} gradg V using interval arithmetic
# ============================================================

M = np.empty((n, n), dtype=object)

for i in range(n):
    for j in range(n):
        total = iv.mpf([0, 0])
        for k in range(n):
            for l in range(n):
                total += V_inv[i, k] * gradg[k, l] * V_iv[l, j]
        M[i, j] = total


# ============================================================
# Compute L_infinity norm (max row sum)
# ============================================================

max_rowsum = 0
for i in range(n):
    rowsum = 0
    for j in range(n):
        rowsum += interval_abs_sup(M[i, j])
    max_rowsum = max(max_rowsum, rowsum)


# ============================================================
# Final outputs
# ============================================================

print(f"M: {max_rowsum}")
print("Err_{800} =", float(er_up.b) + float(er_low.b))
print("E_max =", max(max_V_inv, max_V))

Dpsi = max_rowsum + 5 * (float(er_low.b) + float(er_up.b)) * max(max_V_inv, max_V)

print(
    "L_infinity norm of tilde E^(-1) * DPsi * tilde E after 500 iterations "
    f"is upper bounded by {Dpsi}")

print("Contraction verified:", Dpsi < 1)

\end{verbatim}

\end{appendices}
\end{document}